\documentclass[
final
]{dmtcs-episciences}

\usepackage[utf8]{inputenc}
\usepackage{subfigure}

\usepackage{bm}
\usepackage{amsfonts,amssymb,amsmath,amsthm,epsfig,euscript,youngtab}
\usepackage[margin=1in]{geometry}

\usepackage{thmtools}
\usepackage{thm-restate}
\usepackage{hyperref}
\usepackage{cleveref}

\newtheorem{theorem}{Theorem}
\newtheorem{lemma}[theorem]{Lemma}

\newtheorem{definition}[theorem]{Definition}

\newtheorem{claim}[theorem]{Claim}

\newcommand{\set}[1]{\left\{#1\right\}}

\renewcommand{\bar}[1]{\overline{#1}}

\newcommand{\la}{\lambda}

\newcommand{\N}{\mathbb{N}}
\newcommand{\Z}{\mathbb{Z}}

\usepackage{ mathrsfs, amsmath, amsthm, amscd, amsfonts,
amssymb, graphicx, color, tikz, framed}
\usetikzlibrary{matrix,arrows,patterns}
\numberwithin{equation}{section}

\newcommand{\redx}{{\color{red} \times}}
\newcommand{\bluex}{{\color{blue} \times}}
\newcommand{\greenx}{{\color{green} \times}}
\newcommand{\redo}{{\color{red} \circ}}

\newcommand{\st}[1]{\overset{*}{#1}}
\newcommand{\un}[1]{\underline{#1}}

\everymath=\expandafter{\the\everymath\displaystyle} 

\usepackage[round]{natbib}

\author{Sittipong Thamrongpairoj \affiliationmark{1}
  \and Jeffrey B. Remmel \affiliationmark{2}}

\title{Positional Marked Patterns in Permutations}
\affiliation{
  Faculty of Science at Sriracha, Kasetsart University Sriracha Campus\\
  Department of Mathematics, University of California San Diego}
\received{2021-02-12}

\revised{2021-10-01, 2022-05-31}

\accepted{2022-06-02}

\publicationdetails{24}{2022}{1}{23}{7171}
\begin{document}
\maketitle
\begin{abstract}
We introduce the notion of positional marked pattern, which is a permutation $\tau$ with one element underlined. Given a permutation $\sigma$, we say that $\sigma$ has a $\tau$-match at position $i$ if $\tau$ occurs in $\sigma$ in such a way that $\sigma_i$ plays the role of the underlined element in the occurrence. We let $pmp_\tau(\sigma)$ denote the number of positions $i$ which $\sigma$ has a $\tau$-match. This defines a new class of statistics on permutations, where we study such statistics and prove a number of results. In particular, we prove that two positional marked patterns $1\underline{2}3$ and $1\underline{3}2$ give rise to two statistics that have the same distribution. The equidistibution phenomenon also occurs in other several collections of patterns like $\left \{1\underline{2}3 , 1\underline{3}2 \right \}$, and $\left \{ 1\un234, 1\un243, \un2134, \un2 1 4 3 \right \}$, as well as two positional marked patterns of any length $n$: $\left \{ 1\underline 2\tau , \underline 21\tau \right \}$.
\end{abstract}

\section{Introduction}\label{intro}

The notion of mesh patterns was first introduced by \index{Br\"{a}nd\'en} \index{Claesson} Br\"{a}nd\'en and Claesson \cite{BC}. Many  authors have further studied this notion. In particular, the notion of marked mesh pattern was introduced by \index{\'Ulfarsson} \'Ulfarsson \cite{U}, and the study of the distributions of quadrant marked mesh patterns in permutations was initiated by \index{Kitaev} \index{Remmel} Kitaev and Remmel in \cite{KR}.

Here, we recall the definition of a quadrant \index{Marked mesh pattern} marked mesh pattern. Let $\N = \set{0,1,2,\ldots}$ denote the set of natural numbers and $S_n$ denote the symmetric group of permutations of $1,\ldots,n$. We consider the \index{Graph of permutation} graph of $\sigma, G(\sigma)$, to be the set of points $(i,\sigma_i)$. We are interested in the points that lie in the four quadrants I, II, III, IV of that coordinate system. For any $a,b,c,d \in \N$ and any $\sigma = \sigma_1 \ldots \sigma_n \in S_n$, we say that $\sigma_i$ matches {\em the simple marked mesh pattern} $MMP(a,b,c,d)$ if in $G(\sigma)$ relative to the coordinate system which has the point $(i,\sigma_i)$ as its origin, there are at least $a$ points in quadrant I, at least $b$ points in quadrant II, at least $c$ points in quadrant III, and at least $d$ points in quadrant IV. We let $mmp^{(a,b,c,d)}(\sigma)$ denote the number of $i$ such that $\sigma_i$ matches the marked mesh pattern $MMP(a,b,c,d)$ in $\sigma$.


For example, let $\sigma = 647913258$ and consider simple marked mesh pattern \allowbreak $MMP(2,0,2,0)$. We look for points $(i,\sigma_i)$ in the graph of $\sigma$ such that there are at least 2 points to the top right of $(i,\sigma_i)$ and at least 2 points to the bottom left of $(i,\sigma_i)$. The graph of $\sigma$ and an illustration of $MMP(2,0,2,0)$ are shown in Figure \ref{exmmp}. As we see in Figure \ref{exmmpat7}, at point $(3,7)$, there are two points to the top right and 2 points to the bottom left of point $(3,7)$, so $7$ matches $MMP(2,0,2,0)$ in $\sigma$. At point $(6,3)$, on the other hand, there are 2 points to the top right but only 1 point to the bottom left of point $(6,3)$. Therefore, $3$ does not match $MMP(2,0,2,0)$ in $\sigma$. In this case, only 7 matches $MMP(2,0,2,0)$ in $\sigma$, so $mmp^{(2,0,2,0)}(\sigma) = 1$. Quadrant marked mesh pattern was studied further in \cite{MD}, \cite{KR2}, \cite{KR3}, \cite{KRT}, \cite{KRT2} and \cite{QR}.
\begin{figure}[h]
\begin{center}
\begin{tikzpicture}[scale=.5]
\draw [<->] (-1,0) -- (10,0);
\draw [<->] (0,-1) -- (0,10);
\filldraw (1,6) circle (2pt);
\filldraw (2,4) circle (2pt);
\filldraw (3,7) circle (2pt);
\filldraw (4,9) circle (2pt);
\filldraw (5,1) circle (2pt);
\filldraw (6,3) circle (2pt);
\filldraw (7,2) circle (2pt);
\filldraw (8,5) circle (2pt);
\filldraw (9,8) circle (2pt);
\draw [<->] (15,2) -- (15,8);
\draw [<->] (12,5) -- (18,5);
\node at (16.5,6.5) {$\ge 2$};
\node at (12.5,3.5) {$\ge 2$};
\end{tikzpicture}
\end{center}
\caption{Example for marked mesh pattern.}
\label{exmmp}
\end{figure}
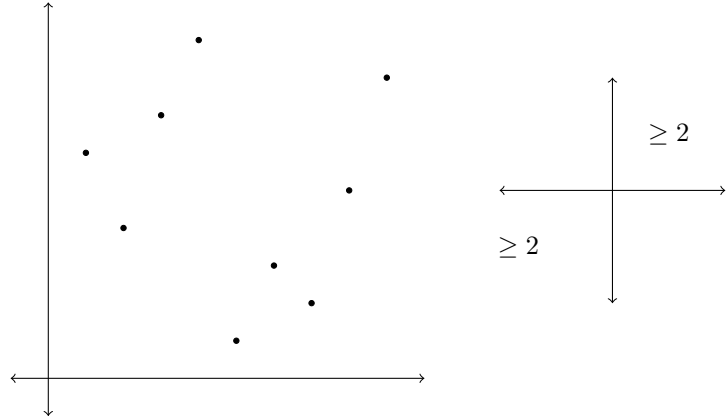

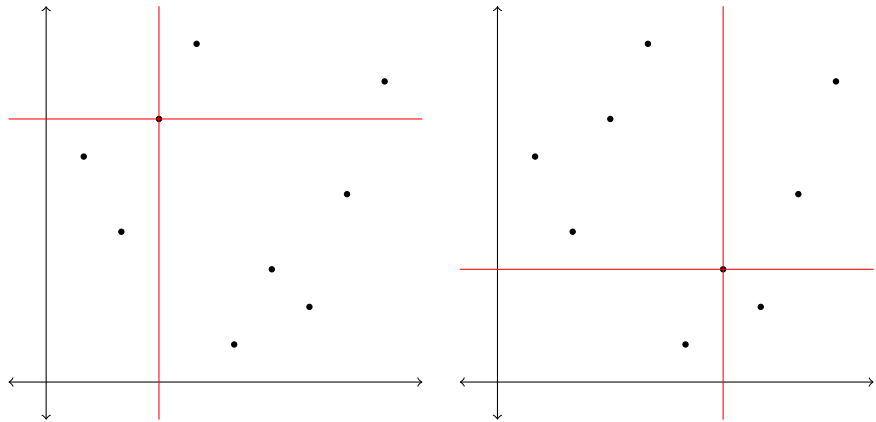
\begin{figure}
\begin{center}
\begin{tikzpicture}[scale=.5]
\draw [<->] (-1,0) -- (10,0);
\draw [<->] (0,-1) -- (0,10);
\filldraw (1,6) circle (2pt);
\filldraw (2,4) circle (2pt);
\filldraw (3,7) circle (2pt);
\filldraw (4,9) circle (2pt);
\filldraw (5,1) circle (2pt);
\filldraw (6,3) circle (2pt);
\filldraw (7,2) circle (2pt);
\filldraw (8,5) circle (2pt);
\filldraw (9,8) circle (2pt);
\draw[red] (3,-1) -- (3,10);
\draw[red] (-1,7) -- (10,7);

\draw [<->] (11,0) -- (22,0);
\draw [<->] (12,-1) -- (12,10);
\filldraw (13,6) circle (2pt);
\filldraw (14,4) circle (2pt);
\filldraw (15,7) circle (2pt);
\filldraw (16,9) circle (2pt);
\filldraw (17,1) circle (2pt);
\filldraw (18,3) circle (2pt);
\filldraw (19,2) circle (2pt);
\filldraw (20,5) circle (2pt);
\filldraw (21,8) circle (2pt);
\draw[red] (18,-1) -- (18,10);
\draw[red] (11,3) -- (22,3);
\end{tikzpicture}
\end{center}
\caption{7 matches $MMP(2,0,2,0)$ in $\sigma$ but 3 does not.}
\label{exmmpat7}
\end{figure}

We studied a generalization of quadrant marked mesh patterns. Here, we make a connection between quadrant marked mesh pattern and classical permutation patterns. Let $\sigma = \sigma_1 \sigma_2 \ldots \sigma_n \in S_n$ and $\tau = \tau_1\tau_2\ldots \tau_k \in S_k$. We say that $\sigma$ contains pattern $\tau$ if there is a subsequence $\sigma' = \sigma_{i_1} \sigma_{i_2} \ldots \sigma_{i_k}$ of $\sigma$ such that $\sigma'$ and $\tau$ have the same relative order. That is, $\sigma_{i_j} < \sigma_{i_k}$ if and only if $\tau_j < \tau_k$.  For example, permutation $15324$ contains pattern $123$ as the subsequence $134$ (and $124$) has the same relative order as pattern $123$. We now consider the pattern $MMP(1,0,1,0)$. An element in a permutation matches $MMP(1,0,1,0)$ when there is at least one point to the top right and at least one point to the bottom left of the interested point. An equivalent way to think about $MMP(1,0,1,0)$ for any permutation $\sigma$ is to count the number of elements in $\sigma$ that can be the midpoint of the pattern $123$ in $\sigma$. One can ask similar questions but for different patterns and positions. For example, one may consider the number of elements in $\sigma$ that can be the starting point of the pattern $2314$. 

Here, we briefly define the notion of {\em  positional marked patterns (PMP)} \index{Positional marked pattern} which enumerates the described statistics. A positional marked pattern of length $k$ is a permutation of $[k]$ with one of the elements underlined. Given any positional marked pattern $\tau$, let $\pi(\tau)$ denote the element in $S_k$ obtained from $\tau$ by removing the underline, and let $u(\tau)$ be the position of the underlined element in $\tau$. Given $\sigma = \sigma_1 \sigma_2 \ldots \sigma_n \in S_n$ and a positional marked pattern $\tau$, we say that $\sigma$ has a $\tau$-match at position $\ell$ if $\sigma$ contains the pattern $\pi(\tau)$ in such a way that the $\ell$-th element in $\sigma$ plays the role of the underlined elements in $\tau$. Let $pmp_\tau(\sigma)$ denote the number of positions $\ell$ such that $\sigma$ has a $\tau$-match at position $\ell$. We will carefully define positional marked patterns in Section \ref{chapdef}.

Recall that, in classical patterns, two patterns $\tau_1$ and $\tau_2$ are \index{Classical Wilf-equivalent} Wilf-equivalent if the number of $\sigma$ in $S_n$ avoiding $\tau_1$ is the same as that of $\tau_2$ for all $n \in \N$. Here, we adopt the vocabulary from the classical definition to our definition. Given two positional marked patterns $\tau_1$ and $\tau_2$, we say that $\tau_1$ and $\tau_2$ are \index{$PMP$-Wilf-equivalent}{\em $PMP$-Wilf-equivalent} if $pmp_{\tau_1}$ and $pmp_{\tau_2}$ have the same distribution over $S_n$ for all $n$. Our main goal is to classify $PMP$-Wilf-equivalent classes for positional marked patterns. 

The paper outlines as follows: In Section \ref{chapdef}, we give a precise definition and some preliminary results on positional marked patterns. In Section \ref{len3}, we study positional marked patterns of length $3$. We prove that there are only two $PMP$-Wilf-equivalent classes for positional marked patterns of length 3. The result follows from the following theorems.

\begin{restatable}{theorem}{threeone}
\label{threeone}
Two positional marked patterns $\un 1 2 3$ and $\un 1 3 2$ are $PMP$-Wilf-equivalent. 
\end{restatable}

\begin{restatable}{theorem}{threetwo}
\label{threetwo}
Two positional marked patterns $1 \un 2 3$ and $1 \un 3 2$ are $PMP$-Wilf-equivalent.
\end{restatable}

In Section \ref{yoyo}, we prove some non-trivial equivalences of pairs of positional marked patterns of length 4 as well as providing numerical results for every patterns of length 4. In Section \ref{lenn}, we prove the following theorem, which gives an equivalence of pairs of positional marked patterns of arbitrary length.

\begin{restatable}{theorem}{lengthn}
\label{lengthn}
Let $n\ge 4$ and  $p_1,p_2,\ldots,p_{n-2}$ be a rearrangement of  $3,4,\ldots n$. Three positional marked patterns $P_1 = 1\un 2 p_1 \ldots p_{\ell-2}$, $P_2 = \un 2 1 p_1 \ldots p_{\ell-2}$ and $P_3 = 2 \un 1 p_1 \ldots p_{\ell-2}$ are $PMP$-Wilf-equivalent.
\end{restatable}

In Section \ref{future}, we discuss further research possibilities, as well as precise connection between positional marked patterns and quadrant marked mesh patterns.

\section{Definition}\label{chapdef}

\begin{definition}
Let $S_k^*$ denote a set of permutations of $[k]$ with one of the elements underlined. Given any $\tau \in S_k^*$, let $\pi(\tau)$ denote an element in $S_k$ obtained from $\tau$ by removing the underline, and let $u(\tau)$ be the position of the underlined element in $\tau$. We shall call an element in $S_k^*$ a positional marked pattern (PMP).
\end{definition}

For example, $\tau = 1 \un{4} 3 2$ is an element in $S_4^*$. In this case, $\pi(\tau) = 1432$ and $u(\tau) = 2$.

\begin{definition}
Given a word $w$ where letters are taken from $\Z_{\ge 0}$, a reduction of $w$, denoted by $red(w)$, is the word obtained from $w$ by replacing the $i$-th smallest letter by $i$. In particular, if $w$ is a word of length $k$ with distinct letters, then $red(w) \in S_k$.
\end{definition}

For example, if $w= 25725$, then $red(w) = 12312$. Here, we are ready to define a statistic on $S_n$ that we mainly focus on in this paper. 

\begin{definition}
Given $\sigma = \sigma_1 \sigma_2 \ldots \sigma_n \in S_n$ and $\tau \in S_k^*$, we say that $\sigma$ has a $\tau$-match at position $\ell$ if there is a subsequence $\sigma_{i_1} \sigma_{i_2} \ldots \sigma_{i_k}$ in $\sigma$ such that

\begin{enumerate}
\item $\ell = i_{u(\tau)}$
\item $red(\sigma_{i_1} \sigma_{i_2} \ldots \sigma_{i_k}) = \pi(\tau)$.
\end{enumerate}
 
In other words, $\sigma$ contains the pattern $\pi(\tau)$ in such a way that $\ell$-th element in $\sigma$ plays the role of the underlined elements in $\tau$. Let $pmp_\tau(\sigma)$ denote the number of positions $\ell$ such that $\sigma$ has a $\tau$-match at position $\ell$.

\end{definition}

For example, if $\sigma = 26481573$ and $\tau = 1\un 4 3 2$. Then $\sigma$ has a $\tau$-match at positions 2 and 4, as we find subsequences $2653$ and $2853$ respectively. Thus, $pmp_\tau(\sigma) = 2$. We are interested in the generating function $P_{n,\tau}(x) = \sum_{\sigma \in S_n} x^{pmp_{\tau}(\sigma)}$. We say that two positional marked patterns $\tau_1$ and $\tau_2$ are $PMP$-Wilf-equivalent if $P_{n,\tau_1}(x) = P_{n,\tau_2}(x)$ for all $n$. Note that, by looking at the constant terms of generating functions, if $\tau_1$ and $\tau_2$ are $PMP$-Wilf-equivalent, then $\pi(\tau_1)$ and $\pi(\tau_2) $ are Wilf-equivalent. Thus, one might think of $pmp$-Wilf-equivalent as a stronger version of Wilf-equivalent. In this paper, we classify equivalence classes of $S_3^*$ and $S_4^*$.

We associate a positional marked pattern with a permutation matrix-like diagram. Given any $\tau \in S_k^*$, the diagram associated to $\tau$ is a $k$ by $k$ array with the following filling: For the cell at $i$-th row and $j$-th column (i) the cell is filled with $\circ$ if $\tau_i = j$ and $u(\tau) = i$, (ii) the cell is filled with $ \times$ if $\tau_i = j$ and $u(\tau) \neq i$, or (iii) the cell is empty otherwise. By convention, we count rows and columns of an array from left to right and from bottom to top. For example, if $\tau = 1\un4 3 2$, then the corresponding diagram is

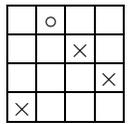
\begin{center} \young(\hfil \circ \hfil \hfil  ,\hfil \hfil \times \hfil,\hfil \hfil \hfil \times,\times \hfil \hfil \hfil ) .\end{center}

Similarly, we associate $\sigma \in S_n$ with an $n$ by $n$ diagram with the cell at $i$-th row and $j$-column is (i) filled with $\times$ if $\sigma_i = j$ and (ii) empty otherwise. For example, the diagram corresponding to $\pi = 26481573$ is

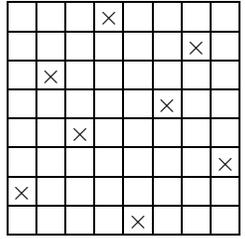
\begin{center} \young(\hfil \hfil \hfil \times \hfil \hfil \hfil \hfil,\hfil \hfil \hfil \hfil \hfil \hfil \times \hfil,\hfil \times \hfil \hfil \hfil \hfil \hfil \hfil,\hfil \hfil \hfil \hfil \hfil \times \hfil \hfil,\hfil \hfil \times \hfil \hfil \hfil \hfil \hfil,\hfil \hfil \hfil \hfil \hfil \hfil \hfil \times,\times \hfil \hfil \hfil \hfil \hfil \hfil \hfil,\hfil \hfil \hfil \hfil \times \hfil \hfil \hfil) . \end{center}

We can visualize  matching of $PMP$ using diagrams. Given a permutation $\sigma$ and a $PMP$ $\tau$, $\sigma$ has a $\tau$-match at position $i$ if when replacing $\times$ in the $i$-th column of the diagram of $\sigma$ by $\circ$, then it contains a subdiagram $\tau$. For example, $\sigma = 26481573$ has a $\tau$-match at position 2 and 4 by looking at subdiagrams marked in red below

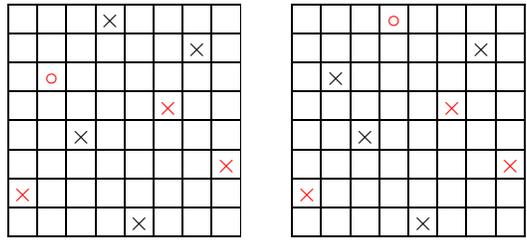
\begin{center} \young(\hfil \hfil \hfil \times \hfil \hfil \hfil \hfil,\hfil \hfil \hfil \hfil \hfil \hfil \times \hfil,\hfil \redo \hfil \hfil \hfil \hfil \hfil \hfil,\hfil \hfil \hfil \hfil \hfil \redx \hfil \hfil,\hfil \hfil \times \hfil \hfil \hfil \hfil \hfil,\hfil \hfil \hfil \hfil \hfil \hfil \hfil \redx,\redx \hfil \hfil \hfil \hfil \hfil \hfil \hfil,\hfil \hfil \hfil \hfil \times \hfil \hfil \hfil) \hspace{.5cm} \young(\hfil \hfil \hfil \redo \hfil \hfil \hfil \hfil,\hfil \hfil \hfil \hfil \hfil \hfil \times \hfil,\hfil \times \hfil \hfil \hfil \hfil \hfil \hfil,\hfil \hfil \hfil \hfil \hfil \redx \hfil \hfil,\hfil \hfil \times \hfil \hfil \hfil \hfil \hfil,\hfil \hfil \hfil \hfil \hfil \hfil \hfil \redx,\redx \hfil \hfil \hfil \hfil \hfil \hfil \hfil,\hfil \hfil \hfil \hfil \times \hfil \hfil \hfil) .\end{center}

With diagrams, we prove equivalences of $PMP$ by symmetry. Given two $PMP$s $\tau_1,\tau_2$ such that the diagram of $\tau_2$ can be obtained from $\tau_1$ by applying a series of rotations and reflections to the diagram of $\tau_1$, then we can construct a map $\theta: S_n \rightarrow S_n$ by applying the same series of rotations and reflections to elements in $S_n$. The map will obviously have the property that $pmp_{\tau_1}(\sigma) = pmp_{\tau_2} (\theta(\sigma))$ for all $\sigma \in S_n$, and so $\tau_1$ and $\tau_2$ are $PMP$-Wilf-equivalent. We proved the following lemma.

\begin{lemma} \label{sym} Given $\tau_1,\tau_2 \in S_k^*$, such that $\tau_2$ can be obtained by applying a series of rotations and reflections to $\tau_1$. Then $\tau_1$ and $\tau_2$ are $PMP$-Wilf-equivalent.
\end{lemma}

Lemma \ref{sym} reduces the problem tremendously. There are $3!\cdot 3 = 18$ $PMP$s of length 3. However, with Lemma \ref{sym}, there are at most 4 equivalence classes, which are represented by $\un1 2 3, 1 \un 2 3, \un 1 3 2, 1 \un 3 2$. Their diagrams are shown in Figure \ref{4replen3}.

\begin{figure}[h]
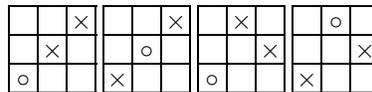

\begin{center}
\young(\hfil \hfil \times,\hfil \times \hfil,\circ \hfil \hfil) \young(\hfil \hfil \times,\hfil \circ \hfil,\times \hfil \hfil) \young(\hfil \times \hfil,\hfil \hfil \times,\circ \hfil \hfil) \young(\hfil \circ \hfil,\hfil \hfil \times,\times \hfil \hfil)
\end{center}
\caption{Diagram shows 4 representatives of positional marked patterns of length 3 modulo geometrical equivalence.}
\label{4replen3}
\end{figure}

\section{Equivalence classes of $S_3^*$}\label{len3}

\subsection{Equivalence of $\un 1 2 3$ and $\un 1 3 2$}

In this subsection, we restate and prove one of the theorems mentioned in Section \ref{intro}.

\threeone*

For our convenience, we let $\tau_1 = \un1 2 3$ and $\tau_2 = \un1 3 2$. We will first prove by showing that generating functions corresponding to those patterns satisfy the same recursive formula. We later will construct a bijection from $S_n$ to itself that maps $pmp_{\tau_1}$ to $pmp_{\tau_2}$, that is, a map $\theta$ such that $pmp_{\tau_1}(\sigma) = pmp_{\tau_2}(\theta(\sigma))$.

Given any permutation $\sigma = \sigma_1\sigma_2\ldots \sigma_n \in S_n$. We say that $\sigma$ has an {\em ascent} at position $i$ if $\sigma_i<\sigma_{i+1}$. We will look at the position where the last ascent occurs. In particular, either $\sigma$  has no ascent or there is $k$ such that $\sigma_k < \sigma_{k+1} > \sigma_{k+2} > \ldots > \sigma_{n}$. In that case, we say that the last ascent of $\sigma$ is at position $k$. Let $P_{n,\tau_1,k}(x) = \sum x^{pmp_{\tau_1}(\sigma)}$ where the sum is over all permutations in $S_n$ with the last ascent at position $k$. Then, we have $P_{n,\tau_1}(x) = 1+ \sum_{k=1}^{n-1} P_{n,\tau_1,k}(x)$.

Given any permutation $\sigma = \sigma_1 \sigma_2 \ldots \sigma_n \in S_n$. We say that $\sigma$ has an {\em descent} at position $i$ if $\sigma_i>\sigma_{i+1}$. Similar to $\tau_1$, we will look at the position where the last descent occurs for $\tau_2$. Either $\sigma$ has no descent, or there is $k$ such that $\sigma_k > \sigma_{k+1} < \sigma_{k+2} < \ldots < \sigma_n$. We say that the last descent of $\sigma$ is at position $k$. Let $P_{n,\tau_2,k}(x) = \sum x^{pmp_{\tau_2}(\sigma)}$ where the sum is over all permutations in $S_n$ with the last descent at position $k$. Then, we have $P_{n,\tau_2}(x) = 1+ \sum_{k=1}^{n-1} P_{n,\tau_2,k}(x)$.

First, we derive a recursive formula for $P_{n,\tau_1,k}(x)$.

\begin{lemma} $P_{n,\tau_1,k}(x)$ satisfies the following recursive formula:

$$P_{n,\tau_1,k}(x) = (k-1)xP_{n-1,\tau_1,k-1}(x) + 1+ \sum_{\ell=1}^{k}P_{n-1,\tau_1,\ell}(x)$$

where $P_{n,\tau_1,0}(x) = P_{n,\tau_1,n}(x) = 0$ by convention.
\end{lemma}

\begin{proof}
We derive a recursive formula of $P_{n,\tau_1,k}(x)$ by looking at position of $1$ in $\sigma$ where the last ascent is at position $k$. So, $\sigma$ has the following form:

$$\sigma = \sigma_1 ~ \sigma_2 \ldots \sigma_{k-1} ~ \sigma_k  <  \sigma_{k+1} > \sigma_{k+2} > \ldots > \sigma_n .$$

 Let $t$ be the position of $1$. We have three cases:
\begin{enumerate}
\item $1\le t \le k-1$. In this case, $1,\sigma_k,\sigma_{k+1}$ form the pattern $\tau_1$. Thus, $\sigma$ has a $\tau_1$-match at position $t$. Moreover, 1 does not influence $\tau_1$-match at other positions. Thus, we can remove 1 from $\sigma$ and decrease other elements by 1. The last ascent will be at position $k-1$. Therefore,
this case contributes $x(k-1)P_{n-1,\tau_1,k-1}(x)$.
\item $t = k$. In this case, $\sigma$ does not have a $\tau_1$-match at position $k$. We can remove $1$ from $\sigma$ and decrease other elements by 1. Either the remaining permutation will has no ascent, or the last ascent appears at some position between 1 and $k-1$. Thus, this case contributes $1+\sum_{\ell=1}^{k-1}P_{n-1,\tau_1,\ell}(x)$.
\item $t = n$. In this case, $\sigma$ does not have a $\tau_1$-match at position $n$. We can remove $1$ from $\sigma$ and decrease other elements by 1. The remaining permutation will have the last ascent at position $k$. Thus, this case contributes $P_{n-1,\tau_1,k}(x)$.
\end{enumerate}

In total, we have

\begin{eqnarray*}
P_{n,\tau_1,k}(x) &=& (k-1)xP_{n-1,\tau_1,k-1}(x) + 1+ \left ( \sum_{\ell=1}^{k-1} P_{n-1,\tau_1,\ell}(x) \right )+ P_{n-1,\tau_1,k}(x)\\
&=& (k-1)xP_{n-1,\tau_1,k-1}(x) + 1+ \sum_{\ell=1}^{k}P_{n-1,\tau_1,\ell}(x) .
\end{eqnarray*}

Note that, the first case does not exist when $k=1$, but the formula is correct since $P_{n-1,\tau_1,0}(x) = 0$. Also, the third case does not exist when $k=n-1$. However, the formula is also correct since $P_{n,\tau_1,n}(x) = 0$.

\end{proof}

Here, we prove a similar result for $\tau_2$.

\begin{lemma} Let $\tau_2 = \un 1 3 2$ and let $P_{n,\tau_2,k}(x) = \sum x^{pmp_{\tau_2}(\sigma)}$ where the sum is over all $\sigma \in S_n$ with the last descent of $\sigma$  at position $k$. Then, $P_{n,\tau_2,k}(x)$ satisfies the following recursive formula:

$$P_{n,\tau_2,k}(x) = (k-1)xP_{n-1,\tau_2,k-1}(x) + 1+ \sum_{\ell=1}^{k}P_{n-1,\tau_2,\ell}(x)$$

where $P_{n,\tau_2,0}(x) = P_{n,\tau_2,n}(x) = 0$ by convention.
\end{lemma}

\begin{proof}

We use the same strategy as in the case for $\tau_1$. We derive a recursive formula for $P_{n,\tau_2,k}(x)$ by looking at the position of $1$ in $\sigma$ where the last descent is at position $k$. So, $\sigma$ has the following form:

$$\sigma = \sigma_1 ~ \sigma_2 \ldots \sigma_{k-1} ~ \sigma_k  >  \sigma_{k+1} < \sigma_{k+2} < \ldots < \sigma_n.$$

 Let $t$ be the position of $1$. We have 2 cases:
\begin{enumerate}
\item $1\le t \le k-1$. In this case, $1,\sigma_k,\sigma_{k+1}$ form the pattern $\tau_2$. Thus, $\sigma$ has a $\tau_2$-match at postion $t$. Moreover, 1 does not influence $\tau_2$-match at other positions. Thus, we can remove 1 from $\sigma$ and decrease other elements by 1. The last descent will be at position $k-1$. Therefore,
this case contributes $x(k-1)P_{n-1,\tau_2,k-1}(x)$.

\item $t = k+1$. In this case, $\sigma$ does not have a $\tau$-match at position $k+1$. We can remove $1$ from $\sigma$ and decrease other elements by 1. The remaining permutation will either has no descent, or the last ascent appears at some position between 1 and $k$. Thus, this case contributes $1+\sum_{\ell=1}^{k}P_{n-1,\tau_1,\ell}(x)$.
\end{enumerate}

In total, we have

$$
P_{n,\tau_2,k}(x) =(k-1)xP_{n-1,\tau_2,k-1}(x) + 1+ \sum_{\ell=1}^{k} P_{n-1,\tau_2,\ell}(x).
$$

Note that, the first case does not exist when $k=1$, but the formula is correct since $P_{n-1,\tau_2,0}(x) = 0$. Also, for the second case, if $k = n-1$, the last descent cannot be at position $k = n-1$. However, the formula is still correct since $P_{n-1,\tau_2,n}(x) = 0$.
\end{proof}

One can check that $P_{2,\tau_1,1}(x) = 1 = P_{2,\tau_2,1}(x)$. Since $P_{n,\tau_1,k}(x)$ and $P_{n,\tau_2,k}(x)$ satisfy the same recursive formula and have the same initial values, we prove that $P_{n,\tau_1,k}(x) = P_{n,\tau_2,k}(x)$ for all $n,k$ such that $1 \le k \le n-1$, and hence $P_{n,\tau_1}(x) = P_{n,\tau_2}(x)$ for all $n \ge 1$. Therefore, we prove Theorem \ref{threeone}.

To conclude this section, we give a bijection $\theta: S_n \rightarrow S_n$ such that $pmp_{\tau_1}(\sigma) = pmp_{\tau_2}(\theta(\sigma))$. The bijection is constructed based on the recursive formula. We start with $\theta(1) = 1$, $\theta(12) = 21$ and $\theta(21) = 12$. In general, $\theta$ will have following properties:

\begin{enumerate}
\item $\sigma$ has no ascent if and only if $\theta(\sigma)$ has no descent.
\item The last ascent of $\sigma$ is at the same position as the last descent of $\theta(\sigma)$.
\item $pmp_{\tau_1}(\sigma) = pmp_{\tau_2}(\theta(\sigma))$.
\end{enumerate}

By observation, $\theta$ satisfies the properties for $S_1$ and $S_2$. For $S_{n+1}$ where $n \ge 2$, we define the map recursively. Given $\sigma \in S_{n}$, we have $\sigma' =  \theta(\sigma) \in S_{n}$, where the last ascent of $\sigma$ is at the same position as the last descent of $\sigma'$. Let $k$ be the position. We decompose $\sigma$ and $\sigma'$ at $k$.

$$\sigma = \sigma_1 \sigma_2 \ldots \sigma_k \hspace{1cm} \sigma_{k+1} \ldots \sigma_n$$
$$\sigma' = \sigma'_1 \sigma'_2 \ldots \sigma'_k \hspace{1cm} \sigma'_{k+1} \ldots \sigma'_n$$

If $\sigma$ has no ascent, then $\sigma'$ has no descent, we decompose $\sigma$ and $\sigma'$ by having the ``first part'' empty, or, equivalently, set $k=0$.

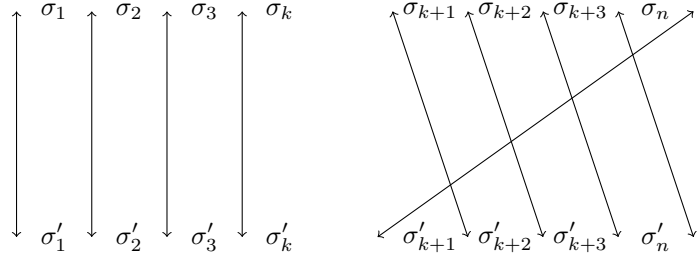
\begin{figure}[h]
\begin{center}
\begin{tikzpicture}
\node at (1,4) {$\sigma_1$};
\node at (2,4) {$\sigma_2$};
\node at (3,4) {$\sigma_3$};
\node at (4,4) {$\sigma_k$};
\node at (6,4) {$\sigma_{k+1}$};
\node at (7,4) {$\sigma_{k+2}$};
\node at (8,4) {$\sigma_{k+3}$};
\node at (9,4) {$\sigma_{n}$};

\node at (1,1) {$\sigma'_1$};
\node at (2,1) {$\sigma'_2$};
\node at (3,1) {$\sigma'_3$};
\node at (4,1) {$\sigma'_k$};
\node at (6,1) {$\sigma'_{k+1}$};
\node at (7,1) {$\sigma'_{k+2}$};
\node at (8,1) {$\sigma'_{k+3}$};
\node at (9,1) {$\sigma'_{n}$};

\draw[<->] (.5,1) -- (.5,4);
\draw[<->] (1.5,1) -- (1.5,4);
\draw[<->] (2.5,1) -- (2.5,4);
\draw[<->] (3.5,1) -- (3.5,4);

\draw[<->] (6.5,1) -- (5.5,4);
\draw[<->] (7.5,1) -- (6.5,4);
\draw[<->] (8.5,1) -- (7.5,4);
\draw[<->] (9.5,1) -- (8.5,4);
\draw[<->] (5.3,1) -- (9.5,4);

\end{tikzpicture}
\end{center}
\caption{Inserting positions  of 1 of the bijection $\theta$ that sends $pmp_{\un123}$ to $pmp_{\un132}$}
\label{bijection1}
\end{figure}

Here, we obtain an element $\hat \sigma \in S_{n+1}$ by increasing every element in $\sigma$ by 1 and inserting 1 at some position, then $\theta(\hat \sigma)$ will be obtained from increasing elements in $\sigma'$ by 1 and inserting 1 at some position based on the position of 1 in $\hat \sigma$. Suppose we insert 1 at position $r$ from the left in $\sigma$, then the position $r'$ of 1 inserted in $\sigma'$ is

$$r' = \left \{ \begin{array}{ll} r & r \le k \\ r+1 & k+1 \le r \le n \\ k+1 & r=n+1 \end{array} \right ..$$

The corresponding positions of 1 can be viewed in Figure \ref{bijection1}.

By observation, the position of the last ascent of $\hat \sigma$ is the same as the position of the last descent of $\hat \sigma$. Also, $pmp_{\tau_1}$ will increase by 1 if and only if 1 is inserted at the first $k$ positions, which is the same condition for $pmp_{\tau_2}$ to increase by 1. Thus, $\theta$ satisfies the conditions.

As an example, we start with $\theta(12) = 21$. Here, the diagram for inserting 1 looks like

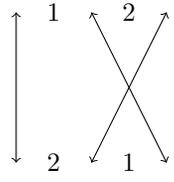
\begin{center}
\begin{tikzpicture}
\node at (1,3) {$1$};
\node at (2,3) {$2$};

\node at (1,1) {$2$};
\node at (2,1) {$1$};

\draw[<->] (0.5,1) -- (0.5,3);
\draw[<->] (1.5,1) -- (2.5,3);
\draw[<->] (2.5,1) -- (1.5,3);
\end{tikzpicture} .
\end{center}

Suppose we insert 1 to 12 at the last position, so we get 231. According to the diagram, we should insert 1 to 21 at the second position, so we get 312. Thus, $\theta(231) =312$. Now the diagram looks like

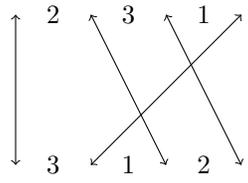
\begin{center}
\begin{tikzpicture}
\node at (1,3) {$2$};
\node at (2,3) {$3$};
\node at (3,3) {$1$};

\node at (1,1) {$3$};
\node at (2,1) {$1$};
\node at (3,1) {$2$};

\draw[<->] (0.5,1) -- (0.5,3);
\draw[<->] (2.5,1) -- (1.5,3);
\draw[<->] (3.5,1) -- (2.5,3);
\draw[<->] (1.5,1) -- (3.5,3);
\end{tikzpicture} .
\end{center}

Here, we insert 1 to 231 at the first position, so we get 1342. We should also insert 1 to 312 at the first position, so we get 1423. Thus, $\theta(1342) = 1423$.

Below are enumerations of $P_{n,\un1 2 3}(x)$ for the first few $n$:

\begin{eqnarray*}
P_{1,\un1 2 3}(x) &=& {1} \\
P_{2,\un1 2 3}(x) &=&  {2}\\
P_{3,\un1 2 3}(x) &=&  {5 + x}\\
P_{4,\un1 2 3}(x) &=&  {14 + 8 x + 2 x^2}\\
P_{5,\un1 2 3}(x) &=&  {42 + 47 x + 25 x^2 + 6 x^3}\\
P_{6,\un1 2 3}(x) &=&  {132 + 244 x + 216 x^2 + 104 x^3 + 24 x^4}\\
P_{7,\un1 2 3}(x) &=&  {429 + 1186 x + 1568 x^2 + 1199 x^3 + 538 x^4 + 120 x^5}\\
P_{8,\un1 2 3}(x) &=&  {1430 + 5536 x + 10232 x^2 + 11264 x^3 + 7814 x^4 + 3324 x^5 +  720 x^6}\\
\end{eqnarray*}

\subsection{Special values for $P_{n,\un 1 2 3}(x)|_{x^k}$}

Even though we do not know a formula for $P_{n,\un 1 2 3 }(x)$ in general, we can still explain some of the coefficients. For example, $P_{n,\un 1 2 3}(x)|_{x^0} = C_n$, the $n$-th Catalan number. This is obvious since a permutation is $123$-avoiding if and only if it does not have a $\un 123$-match at any position. There are other coefficients that have a nice formula.

\begin{theorem}
For $n \ge3$, the degree of $P_{n,\un 1 2 3}(x)$ is $n-2$, and $P_{n,\un 1 2 3}(x)|_{x^{n-2}} = (n-2)!$.
\end{theorem}

\begin{proof}
For any $\sigma \in S_n$, it is clear that $\sigma$ does not have a $\un123$-match at position $n$ or $n-1$. Thus, $pmp_{\un 123 }(\sigma) \le n-2$. For any $\sigma \in S_n$, we claim that $pmp_{\un123} (\sigma) = n-2$ if and only if $\sigma_{n-1} = n-1$ and $\sigma_{n} = n$. The converse obviously true since $\sigma$ would have a $\un123$-match at position $i$ for $1\le i \le n-2$. Suppose that $pmp_{\un123}(\sigma) = n-2$. Let $p_i$ be the position of $i$ in $\sigma$. That is $\sigma_{p_i} = i$. Note that $n-1$ and $n$ cannot be a starting point of the pattern $123$, so $\sigma$ does not have a $\un123$-match at positions $p_{n-1}$ and $p_{n}$. However, $\sigma$ has a $\un123$-match at every position except $n-1$ and $n$. Thus, $\set{p_{n-1},p_n} = \set{n-1,n}$. Consider position $p_{n-2}$. $\sigma$ must be $\un123$-match at position $p_{n-2}$. That is $n-2$ is a starting point of a pattern 123 in $\sigma$. However, the only way for $n-2$ to be a starting point of $123$ is that $n-2,n-1,n$ form the pattern $123$. That is $n-1$ is to the left of $n$. So, $p_{n-1} = n-1$ and $p_n= n$, which means $\sigma_{n-1} = n-1$ and $\sigma_n = n$.

Thus, we know that all $\sigma$ such that $pmp_{\un123}(\sigma) = n-2$ are precisely all $\sigma$ of the form:

$$\sigma = \sigma_1~\sigma_2~\ldots~\sigma_{n-3}~\sigma_{n-2}~n-1~n .$$
There are $(n-2)!$ such $\sigma$'s, so $P_{n,\un123}(x)|_{x^{n-2}} = (n-2)!$.
\end{proof}

This theorem can generalized to any $PMP$.

\begin{theorem}
Given any positive integer $n,k$ such that $n \ge k$, and any $\tau \in S_k^*$, the degree of $P_{n,\tau}(x)$ is $n-k+1$, and $P_{n,\tau}(x)|_{x^{n-k+1}} = (n-k+1)!$.
\end{theorem}

\begin{proof}
Given $k\le n$ and $\tau \in S_k^*$. Suppose $\tau$ has the form

$$\tau = \tau_1 ~ \tau_2 ~ \ldots \un \tau_\ell \ldots ~ \tau_{k-1} ~ \tau_k .$$
That is $u(\tau) = \ell$ and $\pi(\tau) = \tau_1 \tau_2 \ldots \tau_k$. Note that there are $\ell-1$ numbers to the left of the underlined number in $\tau$ and there are $k-\ell$ numbers to the right of the underlined number in $\tau$. Given $\sigma \in S_n$, suppose $\sigma$ has a $\tau$-match at position $i$. Then, there must be at least $\ell-1$ numbers to the left of position $i$ in $\sigma$ and there must be at least $k-\ell$ to the right of position $i$ in $\sigma$. Thus, $\ell \le i \le n-k+\ell$. Therefore, $pmp_{\tau}(\sigma) \le n-k+1$. Thus, the degree of $P_{n,\tau}(x)$ is at most $n-k+1$.

Here, we count the number of $\sigma \in S_n$ such that $pmp_\tau(x) = n-k+1$. We shall prove the following claim:

\begin{claim}
Given $\sigma = \sigma_1 \ldots \sigma_n \in S_n$, $pmp_\tau(\sigma) = n-k+1$ if and only if all of the followings hold:
\begin{enumerate}
\item[{(}1{)}] $\set{\sigma_1,\sigma_2,\ldots,\sigma_{\ell-1},\sigma_{n-k+\ell+1},\ldots,\sigma_n} = \set{1,2,\ldots,\tau_\ell-1,n-k+\tau_\ell+1,\ldots,n}$;
\item[{(}2{)}] $red(\sigma_1\sigma_2\ldots\sigma_{\ell-1}\sigma_{n-k+\ell+1}\ldots\sigma_n) = red(\tau_1\tau_2\ldots\tau_{\ell-1}\tau_{\ell+1}\ldots\tau_k)$.
\end{enumerate}

That is, the first $\ell-1$ numbers and last $k-\ell$ numbers in $\sigma$ is a rearrangement of $ \set{1,2,\ldots,\tau_\ell-1,n-k+\tau_\ell+1,\ldots,n}$ in the way that they have the same relative order as $\tau_1\tau_2\ldots\tau_{\ell-1}\tau_{\ell+1}\ldots\tau_k$.
\end{claim}

\begin{proof}(of the claim)

We first prove the converse. Suppose $\sigma$ satisfies (1) and (2). We want to show that $\sigma$ has a $\tau$-match at any position $i$ when $\ell \le i \le n-k+\ell$. Given any such position $i$. Consider the following subsequence

$$\sigma_1\sigma_2 \ldots \sigma_{\ell-1} ~ \sigma_i ~ \sigma_{n-k+\ell+1} \ldots \sigma_n .$$

Note that $\sigma_i$ is the $\tau_\ell$-th smallest number in the subsequence, and $\tau_\ell$ is the $\tau_\ell$-th smallest number in subsequence $\tau_1 \tau_2 \ldots \tau_k$. Thus, we can insert $\sigma_i$ to $\sigma_1\ldots \sigma_{\ell-1}\sigma_{\ell+1}\ldots\sigma_n$ and $\tau_\ell$ to $\tau_1\ldots\tau_{\ell-1}\tau_{\ell+1}\tau_k$. So, $\sigma_1\sigma_2 \ldots \sigma_{\ell-1} ~ \sigma_i ~ \sigma_{n-k+\ell+1} \ldots \sigma_n$ have the same relative order as $\tau_1\ldots \tau_k$. However, $\tau_1\ldots\tau_k = \pi(\tau)$ is a permutation, so $red(\sigma_1\sigma_2 \ldots \sigma_{\ell-1} ~ \sigma_i ~ \sigma_{n-k+\ell+1} \ldots \sigma_n) = \pi(\tau)$. So, $\sigma$ has a $\tau$-match at position $i$.

Here, we prove the forward direction. Suppose $pmp_\tau(\sigma) = n-k+1$, then $\sigma$ has a $\tau$-match at all positions $i$ when $\ell \le i \le n-k+\ell$, and $\sigma$ does not have a $\tau$-match at all positions $j \in \set{1,2,\ldots,\ell-1, n-k+\ell+1,\ldots,n}$.  Note that, in order for $\sigma$ to match at postion $i$, there must be at least $\tau_\ell-1$ numbers less than $\sigma_i$ and there must be at least $k-\tau_\ell$ numbers greater than $\sigma_i$. Therefore $\tau_\ell \le \sigma_i \le n-k+\tau_\ell$. In other words, $\sigma$ does not have a $\tau$-match at positions of $1,2,\ldots,\tau_\ell-1$ nor positions of $n-k+\tau_\ell , n-k+\tau_\ell+1, \ldots,n$. There are $k-1$ such positions, therefore $\set{p_1,p_2,\ldots,p_{\tau_\ell-1}, p_{n-k+\tau_\ell+1},\ldots , p_n} = \set{1,2,\ldots,\ell-1, n-k+\ell+1,\ldots,n}$, or in other words, $\set{\sigma_1,\sigma_2,\ldots,\sigma_{\ell-1},\sigma_{n-k+\ell+1},\ldots,\sigma_n} = \set{1,2,\ldots,\tau_\ell-1,n-k+\tau_\ell+1,\ldots,n}$. Thus, we prove (1).

To prove (2), we only need to show that there is at most one rearrangement of  $\{ 1,2,\ldots,\tau_\ell-1,n-k+\tau_\ell+1,\ldots,n \}$ such that $pmp_\tau(\sigma) = n-k+1$. Then, by converse direction, we know that the rearrangement in (2) make $pmp_\tau(\sigma) = n-k+1$. In that case, we conclude that the unique rearrangement that make $pmp_{\tau}(\sigma) = n-k+1$ exists and must be (2).

Let's call numbers appearing at positions $1,2,\ldots, \ell-1$ the {\em left part of $\sigma$}, while call the numbers appearing at positions $n-k+\ell+1, n-k+\ell+2,\ldots,n$ the {\em right part of $\sigma$}. Then, $\sigma$ has a following structure:

\vspace{.5cm}

\begin{center}
\begin{tikzpicture}
\node at (4,5) {$1,2,\ldots,\tau_\ell-1,n-k+\tau_\ell+1,\ldots n$};
\draw (0,0.2) -- (0,0) -- (2,0) -- (2,0.2);
\draw (2.5,0.2) -- (2.5,0) -- (5.5,0) -- (5.5,0.2);
\draw (6,0.2) -- (6,0) -- (8,0) -- (8,0.2);
\draw[->] (4,4.5) -- (1.2,.7);
\draw[->] (4,4.5) -- (6.8,.7);
\node at (-1,0) {$\sigma$ $=$};
\node at (1,-.5) {left part};
\node at (7,-.5) {right part};
\end{tikzpicture}.
\end{center}

\vspace{.7cm}

We shall prove that ways to rearrange $1,2,\ldots,\tau_\ell-1, n-k+\tau_\ell+1,\ldots n$ in left and right part of $\sigma$ so that $pmp_\tau(\sigma) = n-k+1$ is unique if exist.

Consider the number $\tau_\ell$ in $\sigma$. We know that $\sigma$ has a $\tau$-match at position $p_{\tau_\ell}$. Then, numbers $1,2, \ldots, \tau_\ell-1$ in $\sigma$ must involve in $\tau$-match at position $p_{\tau_\ell}$ since we need $\tau_\ell-1$ numbers smaller than $\tau_\ell$. Therefore, for each number $i$ less than $\tau_\ell$, we can determine whether $i$ is in the left part or the right part of $\sigma$ based on the relative postion of $i$ and $\tau_\ell$ in $\tau$. Also, consider the number $n-k+\tau_\ell$. $\sigma$ must be a $\tau$-match at position $p_{n-k+\tau_\ell}$. By the same reasoning, for each number $n-k+\tau_\ell+i$, we determine whether $n-k+\tau_\ell+i$ is in the left part or right part of $\sigma$ based on the relative position of $\tau_\ell+i$ and $\tau_\ell$ in $\tau$. Thus, we determine both left and right part as sets.

Now, consider position $\ell$ in $\sigma$. $\sigma$ must have a $\tau$-match at position $\ell$. Thus, numbers in first $\ell-1$ positions in $\sigma$ must involve in a $\tau$-match at position $\ell$, and so, the first $\ell-1$ numbers in $\sigma$ must have the same relative order as first $\ell-1$ elements in $\tau$. By the same reasoning, by considering position $n-k+\ell$, the last $k-\ell$ numbers in $\sigma$ have the same relative order as the last $k-\ell$ numbers in $\tau$. Therefore, we completely determine both left and right part of $\sigma$. Thus, there is at most one way to rearrange  $\set{1,2,\ldots,\tau_\ell-1,n-k+\tau_\ell+1,\ldots,n}$. 

Since the rearrangement in (2) makes $pmp_\tau(\sigma) = n-k+1$, $\sigma$, then $\sigma$ must satisfies (2).

\end{proof}

As an example, if $\tau =164\un352$, any $\sigma \in S_9$ with $pmp_{\tau}(\sigma) = 9-6+1 = 4$ must have the following form:

$$\sigma = 1 ~ 9~ 7 ~~ \sigma_4 ~\sigma_5 ~ \sigma_6 ~\sigma_7 ~ 8 ~ 2.$$

There are $(n-k+1)!$ ways to rearrange the ``middle'' part of $\sigma$. Thus, $P_{n,\tau}(x)|x^{n-k+1} = (n-k+1)!.$

\end{proof}

Another coefficient we can describe is $P_{n,\un1 3 2}(x)|_x$. The sequences is A029760 and A139262 on OEIS. The sequence A139262 counts the sum of all inversion of all elements in $S_n(132)$, the set of 132-avoiders in $S_n$. 

\begin{theorem}\label{suminv}
$P_{n+1,\un1 3 2}(x)|_x = \sum_{\sigma \in S_n(132)} inv(\sigma)$
\end{theorem}

To prove the theorem, we construct sets whose cardinality are $\sum_{\sigma \in S_n(132)} inv(\sigma)$.

\begin{definition}
Let $IMS_n(132)$ be the set of $132$-avoiding $\sigma$ that a pair of elements causing an inversion are marked with $*$.
\end{definition}
As an example, $IMS_3 (132) $ contains 8 elements, which are $\st 2 \st 13$, $\st2 3 \st 1$, $2 \st 3 \st1$, $\st3 \st1 2$, $\st 3 1 \st 2$, $\st 3 \st 2 1$, $\st 3 2 \st 1$, and $3 \st 2 \st 1$. It is easy to see that $|IMS_n(132)| = \sum_{\sigma \in S_n(132)} inv(\sigma)$.

\begin{definition}
Given $\sigma \in S_n$, $i,j \in \Z$, $1 \le i \le n$, $1 \le j \le n+1$. Let $\theta(\sigma,i,j)$ be an element in $S_{n+1}$ obtained from $\sigma$ by 
\begin{enumerate}
\item increasing every number greater than or equal to $i$ by $1$;
\item inserting $i$ at the $j$-th position from the left. (The first position is in front of the first element, and the $n+1$-th position is behind the last element).
\end{enumerate}
\end{definition}

For example, to find $\theta(2143,3,2)$, first, we increase 3,4 by 1 so we have $2154$. Then, we insert 3 at the second position, so we have $23154$. Thus, $\theta(2143,3,2) = 23154$.

Here, we are ready to define a map $\Phi: IMS_n(132) \rightarrow \set{\sigma \in S_{n+1} \hspace{0.1cm} | \hspace{0.1cm} pmp_{\un1 3 2} (\sigma)=1}$.

\begin{definition} 
Given $\bar \sigma \in IMS_n(132)$. Let $\sigma$ be the underlying permutation of $\bar \sigma$. Let $j$ be the position of the first * in $\sigma$, and let $i$ be the number underneath the second * in $\sigma$. Define $\Phi(\bar \sigma) = \theta(\sigma,i,j)$.
\end{definition}

\begin{lemma}
Given $\bar \sigma \in IMS_n(132)$ and $\hat \sigma = \Phi(\bar \sigma)$, then $pmp_{\un 1 3 2}(\hat \sigma) = 1$. Thus, $\Phi$ is a map from $IMS_n(132)$ to $\set{\sigma \in S_{n+1} | pmp_{\un1 3 2} (\sigma)=1}$.
\end{lemma}

\begin{proof}
Let $\bar \sigma \in IMS_n(132)$ has a form

$$\bar \sigma = \sigma_1 \ldots \sigma_{k-1} \st \sigma_k \sigma_{k+1} \ldots \sigma_{\ell-1} \st \sigma_\ell \sigma_{\ell+1} \ldots \sigma_n$$
and $\sigma$ is the underlying permutation. Then, $\Phi(\bar \sigma) = \theta(\sigma , \sigma_l,k)$ looks like

$$\Phi(\bar \sigma) = \sigma_1' \ldots \sigma_{k-1}' \sigma_l  \sigma_k' \sigma_{k+1}' \ldots \sigma_{\ell-1}'\sigma_\ell' \sigma_{\ell+1}' \ldots \sigma_n'$$

where

$$\sigma_i' = \left \{ \begin{array}{ll} \sigma_i +1& \text{for } \sigma_i \ge \sigma_l \\ \sigma_i & \text{for } \sigma_i < \sigma_l \end{array}\right ..$$

Note that $\sigma_\ell' = \sigma_\ell+1$. Also, since $\sigma_k$ and $\sigma_\ell$ cause an inversion, so $\sigma_k > \sigma_\ell$, and thus $\sigma_k' = \sigma_k+1$. Therefore, $\sigma_\ell < \sigma_\ell' < \sigma_k'$. Equvialently, $\sigma_\ell, \sigma_k', \sigma_\ell'$ form the pattern $132$ in $\hat \sigma$. Thus, $pmp_{\un1 3 2}(\hat \sigma) \ge 1$. To prove that $pmp_{\un 1 3 2}(\hat \sigma) = 1$, we need to show that $\hat \sigma$ does not contain a pattern 132 that starts at any element other than $\sigma_\ell$.

Suppose otherwise. That is, $\hat \sigma$ contains the pattern $132$ where the starting element is $\sigma'_t$ for some $t$. We have two cases:

\begin{enumerate}
\item If the pattern does not involve $\sigma_\ell$, then there must be $u,v$ such that $\sigma'_t, \sigma'_u, \sigma'_v$ form a pattern 132. It is easy to see that $\sigma_t, \sigma_u,\sigma_v$ also form a pattern 132 in $\sigma$, which is a contradiction since $\sigma$ is a 132-avoider.
\item If the pattern involves $\sigma_\ell$, then $t<k$ and  $\sigma_t' < \sigma_\ell$. In this case, it is easy to see that $\sigma_t', \sigma_k', \sigma_\ell'$ also form the pattern $132$. Thus, $\sigma_t, \sigma_k, \sigma_\ell$ form the pattern $132$ in $\sigma$, which is again a contradiction.
\end{enumerate}

Thus, $\hat \sigma$ does not contain the pattern 132 that starts at any element except $\sigma_\ell$. So, $pmp_{\un 1 3 2}(\hat \sigma)=1$.
\end{proof}

The map $\Phi$ is clearly an injection. Let $\sigma = \sigma_1 \ldots \sigma_n$ be in the image of $\Phi$, and let $\sigma_i$ be the starting point of the pattern 132. Then, $i$ determines the first starred element, and $\sigma_i$ determines the second starred element. To prove that the map $\Phi$ is a surjective, we need another lemma:

\begin{lemma}
\label{lem5}
Given $\sigma = \sigma_1 \ldots \sigma_n \in S_{n}$ such that $pmp_{\un 1 3 2}(\sigma) = 1$. Suppose the pattern $132$ in $\sigma$ starts at $\sigma_t$, then the following must be true:
\begin{enumerate}
\item[(1)] $\sigma_{t+1} \ge \sigma_t$.
\item[(2)] $\sigma_t +1$ is on the right of $\sigma_t$.
\item[(3)] $\sigma_{t+1} \neq \sigma_{t}+1$.
\end{enumerate}

That is, the graph of $\sigma$ should look like 

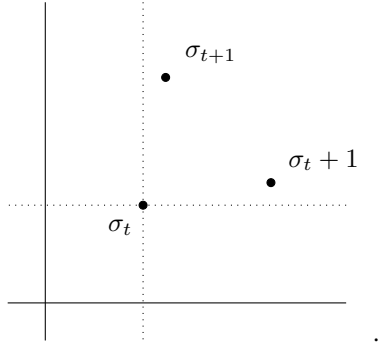
\begin{center}
\begin{tikzpicture}
\draw (4,0) -- (-.5,0);
\draw (0,4) -- (0,-.5);
\node at (1,1) {$\sigma_t$};
\filldraw (1.3,1.3) circle (1.5pt);
\draw [dotted] (4,1.3) -- (-.5,1.3);
\draw [dotted] (1.3,4) -- (1.3,-.5);
\filldraw (1.6,3) circle (1.5pt);
\filldraw (3,1.6) circle (1.5pt);
\node at (2.2,3.3) {$\sigma_{t+1}$};
\node at (3.7,1.9) {$\sigma_t+1$};
\end{tikzpicture} .
\end{center}

\end{lemma}

\begin{proof}
Suppose (1) is not true. There must be $u,v$ such that $\sigma_t,\sigma_u,\sigma_v$ form a pattern 132. Then, $\sigma_{t+1}, \sigma_u, \sigma_v$ also form a pattern 132. So, $pmp_{\un 132}(\sigma) \ge 2$.

Suppose (2) is not true. There must be $u,v$ such that $\sigma_t,\sigma_u,\sigma_v$ form a pattern 132. Then, $\sigma_t+1, \sigma_u, \sigma_v$ also form a pattern 132. So, $pmp_{\un 132}(\sigma) \ge 2$.

Suppose (3) is not true. There must be $u,v$ such that $\sigma_t, \sigma_u, \sigma_v$ form a pattern 132. Then, $\sigma_{t+1}, \sigma_u, \sigma_v$ also form a pattern 132. So, $pmp_{\un 132}(\sigma) \ge 2$.

In all cases, we have a contradiction. Therefore,  (1), (2) and (3) are true. 
\end{proof}

\begin{lemma}
The map $\Phi: IMS_n(132) \rightarrow \set{\sigma \in S_{n+1} | \un1 3 2 (\sigma)=1}$ is surjective.
\end{lemma}

\begin{proof}
Given any $\sigma \in \set{\sigma \in S_{n+1} | \un1 3 2 (\sigma)=1}$. Let $\sigma_t$ be the starting point of the pattern 132 in $\sigma$. By lemma \ref{lem5}, we know that $\sigma_t, \sigma_{t+1}, \sigma_t+1$ form a pattern 132. Let $\tilde \sigma$ be a permutation obtained from $\sigma$ by removing $\sigma_t$ and reducing. Let $\tilde \sigma^*$ be $\tilde \sigma$ with * marked on elements corresponding to $\sigma_{t+1}$ and $\sigma_t+1$ before reducing ($\sigma_{t+1}-1, \sigma_t$ after reducing). $\tilde \sigma^*$ is an element is $IMS_n(132)$ since $\sigma_{t+1},\sigma_t+1$ form an inversion. Clearly, $\Phi(\tilde \sigma^*) = \sigma$.
\end{proof}

For example, consider $\sigma = 785269314 \in S_9$ with $pmp_{132}(\sigma) = 1$, where only possible starting point of the pattern 132 is 2. So, in this example, $\sigma_t, \sigma_{t+1}, \sigma_{t} + 1$ are 2,6,3. Then, removing 2 and reducing give $\tilde \sigma = 67458213$. Then, put * at elements corresponding to 6,3 before reducing, which are 5,2 after reducing. So, $\tilde \sigma^* = 674 \st 5 8 \st 2 13$.

Therefore $\Phi$ is a bijection, and so we prove Theorem \ref{suminv}.

\subsection{Equivalence of $1\un 2 3$ and $1 \un 3 2$}

By observing diagrams of $1\un 3 2$ and $13\un2$, we see that those two patterns are equivalent by Lemma \ref{sym}. Thus, we will instead prove that $1 \un 2 3$ and $1 3 \un 2 $ are equivalent.

\begin{center}
\young(\hfill \circ \hfill,\hfill \hfill \times,\times \hfill \hfill) \young(\hfill \times \hfill,\hfill \hfill \circ,\times \hfill \hfill)

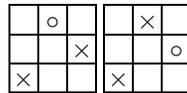
Diagrams of $1 \un 3 2$ and $1 3 \un2$
\end{center}

Here, we restate and prove one of the theorems mentioned in Section \ref{intro}.

\threetwo*

For our convenience, we let $\tau_1 = 1 \un 2 3$ and $\tau_2 =13 \un 2$. We will first prove by showing that two generating functions satisfy the same recursive formula. We later will construct a bijection from $S_n$ to itself that maps an $pmp_{\tau_1}$ to $pmp_{\tau_2}$, that is, a map $\theta$ such that $pmp_{\tau_1}(\sigma) = pmp_{\tau_2}(\theta(\sigma))$. 	

Given any permutation $\sigma \in S_n$. We will look at the position of $1$ in $\sigma$. Let $P_{n,\tau_1,k}(x) = \sum x^{pmp_{\tau_1}(\sigma)}$, where the sum is over all permutations in $S_n$ with 1 at position $k$. Then, we have $P_{n,\tau_1}(x) = \sum_{k=1}^{n} P_{n,\tau_1,k}(x)$.

Here, we derive a recursive formula for $P_{n,\tau_1,k}(x)$.

\begin{lemma} \label{firstcase} Let $\tau_1 = 1\un2 3$ and let $P_{n,\tau_1,k} = x^{pmp_{\tau_1}(\sigma)}$ where the sum is over all $\sigma \in S_n$ with $1$ at position $k$. Then, $P_{n,\tau_1,k}(x)$ satisfies the following recursive formula:

$$P_{n,\tau_1,k}(x) =  \left ( \sum_{\ell=1}^{k-1} P_{n-1,\tau_1,\ell}(x) \right )  + (x (n-k-1) + 1) P_{n-1,\tau_1,k}(x)$$
where $P_{n,\tau_1,N}(x)= 0$ for $N>n$ by convention.
\end{lemma}

\begin{proof}
We derive a recursive formula for $P_{n,\tau_1,k}(x)$ by looking at the position of 2 in $\sigma$ where the position of $1$ is $k$. $\sigma$ has a following form

$$\sigma = \sigma_1 ~ \sigma_2 ~ \ldots ~\sigma_{k-1} ~ 1 ~ \sigma_{k+1} ~\ldots ~\sigma_n .$$

Let $\ell$ be the position of $2$. We have 3 cases:

\begin{enumerate}
\item $1 \le \ell \le k-1$. In this case, we can remove $1$ from $\sigma$ and decrease other elements by 1 without effecting $\tau_1$-match at other position. To see this, first note that if $\sigma$ does not have a $\tau_1$-match at a particular position, removing 1 and reducing will not change it. Also, if $\sigma$ does not have a $\tau_1$-match at some position before 1, then removing $1$ will not change $\tau_1$-match at the position, since $1$ can only serve as $1$ in $1 \un 2 3$. However, $1$ appears after the position that $\tau_1$-match occurs. Thus, $\sigma$ has a $\tau_1$-match at the position without considering 1. Lastly, if $\sigma$ has a $\tau_1$-match at some position $t$ after 1. Note that 2 appears before 1 in $\sigma$. Thus, 2 can serve as 1 in $1 \un 2 3$. Therefore, $\sigma$ still has a $\tau_1$-match at position $t$ when not considering $1$. So, we can remove $1$ and decrease other elements by 1 without effecting $\tau_1$-match. 

After removing 1 and reducing, the position of $1$ will be $\ell$. Thus, this case contributes $\sum_{\ell=1}^{k-1} P_{n-1,\tau_1,\ell}(x)$.

\item $k+1 \le \ell \le n-1$. In this case, $\sigma$ has a $\tau_1$-match at position $\ell$ since $1 ~2~ \sigma_n$ form the pattern $\tau_1$. Moreover, by the same reason as in the first case, removing 2 and reducing will not effect $\tau_1$-match at other position.

After removing 2 and reducing, the position of $1$ is still $\ell$. Thus, this case contributes $x(n-k-1)P_{n-1,\tau_1,k}(x)$.

\item $\ell=n$. In this case, $\sigma$ does not have a $\tau_1$-match at position $\ell$, and removing $2$ will not effect $\tau_1$-match at other positions. Thus, this case contributes $P_{n-1,\tau_1,k}(x)$.
\end{enumerate}

In total, we have

\begin{eqnarray*}
P_{n,\tau_1,k}(x) &=&  \left ( \sum_{\ell=1}^{k-1} P_{n-1,\tau_1,\ell}(x) \right )  + x (n-k-1) P_{n-1,\tau_1,k}(x) + P_{n-1,\tau_1,k}(x) \\
&=& \left ( \sum_{\ell=1}^{k-1} P_{n-1,\tau_1,\ell}(x) \right )  + (x (n-k-1) + 1) P_{n-1,\tau_1,k}(x) .
\end{eqnarray*}

Note that, the first case vanishes when $k=1$, but the formula is correct since it would contribute an empty summation. The second case vanishes when $n-1 \le k \le n$, but the formula is still correct since $n-k-1 = 0$ for $k = n-1$ and $P_{n-1,\tau_1,n}(x) = 0$ for $k=n$. The last case also vanishes when $k=n$, but the formula is still correct as $P_{n-1,\tau_1,n}(x) = 0$.

\end{proof}

Here, we prove a similar result of $\tau_2 = 1 3 \un 2$.

\begin{lemma}
Let $\tau_2 = 1 3\un 2 $ and let $P_{n,\tau_2,k}(x) = x^{pmp_{\tau_2} (\sigma)}$ where the sum is over all $\sigma \in S_n$ with 1 at position $k$. Then, $P_{n,\tau_2,k}(x)$ satisfies the following recursive formula:

$$P_{n,\tau_2,k}(x) =  \left ( \sum_{\ell=1}^{k-1} P_{n-1,\tau_2,\ell}(x) \right )  + (x (n-k-1) + 1) P_{n-1,\tau_2,k}(x).$$

\end{lemma}

\begin{proof}
We derive a recursive formula for $P_{n,\tau_2,k}(x)$ by looking at the position of $2$ in $\sigma$ where the position of $1$ is $k$. $\sigma$ has a following form

$$\sigma = \sigma_1 ~ \sigma_2 ~ \ldots ~\sigma_{k-1} ~ 1 ~ \sigma_{k+1} ~\ldots ~\sigma_n. $$

Let $\ell$ be the position of $2$. We have 3 cases:

\begin{enumerate}

\item $1 \le \ell \le k-1$. In this case, by the same reason as the first case in the proof of Lemma \ref{firstcase}, we can remove $1$ from $\sigma$ and decrease other elements by 1 without effecting $\tau_1$-match at other positions.

After removing 1 and reducing, the position of $1$ will be $\ell$. Thus, this case contributes $\sum_{\ell=1}^{k-1} P_{n-1,\tau_2,\ell}(x)$.

\item $\ell=k+1$. In this case, $\sigma$ does not have a $\tau_2$-match at position $\ell$, and removing $2$ will not effect $\tau_2$-match at other positions by the same reason as above. Thus, this case contributes $P_{n-1,\tau_2,k}(x).$

\item $k+2 \le \ell \le n$. In this case, $\sigma$ has a $\tau_2$-match at position $\ell$ since $1 ~\sigma_{k+1} ~ 2$ form the pattern $\tau_2$. Moreover, by the same reason as above, removing 2 and reducing will not effect $\tau_2$-match at other position.

After removing 2 and reducing, the position of $1$ is still at $\ell$. Thus, this case contribute $x(n-k-1)P_{n-1,\tau_2,k}(x)$.

\end{enumerate}

In total, we have

\begin{eqnarray*}
P_{n,\tau_2,k}(x) &=&  \left ( \sum_{\ell=1}^{k-1} P_{n-1,\tau_2,\ell}(x) \right )  + x (n-k-1) P_{n-1,\tau_2,k}(x) + P_{n-1,\tau_2,k}(x) \\
&=& \left ( \sum_{\ell=1}^{k-1} P_{n-1,\tau_2,\ell}(x) \right )  + (x (n-k-1) + 1) P_{n-1,\tau_2,k}(x).
\end{eqnarray*}

Note that the first case vanishes when $k=1$, but the formula is still correct as it contributes an empty summation. The second case vanishes when $k=n$. The formula is still correct as $P_{n-1,\tau_2,n}(x) = 0 $. The last case vanishes when $k=n-1,n$. When $k=n-1$, the formula is correct as $n-k-1= 0$. When $k=n$, the formula is correct as $P_{n-1,\tau_2,n}(x)=0$.

\end{proof}

It is easy to see that $P_{1,\tau_1,1}(x) = 1 = P_{1,\tau_2,1}(x)$. Since, $P_{n,\tau_1,k}(x)$ and $P_{n,\tau_2,k}(x)$ satisfy the same recursive formula and have the same initial values, we prove that $P_{n,\tau_1,k}(x) = P_{n,\tau_2,k}(x)$ for all $n,k$ such that $1 \le k \le n$, and thus $P_{n,\tau_1}(x) = P_{n,\tau_2}(x)$ for all $n \ge 1$. Hence, we prove Theorem \ref{threetwo}.	

\ \\

We also construct a bijection $\theta : S_n \rightarrow S_n$ such that $pmp_{\tau_1}(\sigma) = pmp_{\tau_2}(\theta(\sigma))$ based on the recursive formula. We define $\theta(1) = 1, \theta(12) = 12$ and $\theta(21) = 21$. In general, $\theta$ will satisfies the following properties:

\begin{enumerate}
\item The position of 1 in $\sigma$ is the same as the position of 1 in $\theta(\sigma)$.
\item $pmp_{\tau_1}(\sigma) = pmp_{\tau_2}(\theta(\sigma)).$
\end{enumerate}

$\theta$ satisfies the properties for $S_1$ and $S_2$. For $S_{n+1}$ where $n \ge 2$, we define the map recursively. Given $\sigma \in S_n$, we have $\sigma' = \theta(\sigma)$. We know that the position of 1 in $\sigma$ is the same as the position of 1 in $\sigma'$. Let $k$ be the position.

Here, we obtain an element $\hat \sigma \in S_{n+1}$ by applying one of the following:

\begin{enumerate}
\item Increase every element by 1, and insert 1 at or after position $k+1$.
\item Increase every element except 1 by 1, and insert 2 at or after position $k+1$.
\end{enumerate}

Then, we obtain $\theta(\hat \sigma)$ by applying similar action to $\sigma'$ based on an action applied to $\sigma$:

\begin{enumerate}
\item If 1 was inserted to $\sigma$, then increase every element in $\sigma'$ by 1 and insert 1 at the same position as inserted in $\sigma$.
\item If 2 was inserted to $\sigma$ at position $r$, then increase every element in $\sigma'$ except 1 by 1 and insert 2 at position $r'$ where

$$r' = \left \{ \begin{array}{ll} r+1 & k+1 \le r \le n \\ k+1 & r=n+1 \end{array} \right ..$$
\end{enumerate}
 Corresponding positions when inserting 1 or 2 can be viewed from Figure \ref{bij2-1} and Firgure \ref{bij2-2}.

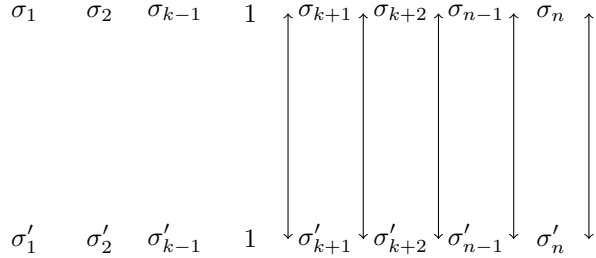
\begin{figure}[h]
\begin{center}
\begin{tikzpicture}
\node at (1,4) {$\sigma_1$};
\node at (2,4) {$\sigma_2$};
\node at (3,4) {$\sigma_{k-1}$};
\node at (4,4) {$1$};
\node at (5,4) {$\sigma_{k+1}$};
\node at (6,4) {$\sigma_{k+2}$};
\node at (7,4) {$\sigma_{n-1}$};
\node at (8,4) {$\sigma_{n}$};

\node at (1,1) {$\sigma'_1$};
\node at (2,1) {$\sigma'_2$};
\node at (3,1) {$\sigma'_{k-1}$};
\node at (4,1) {$1$};
\node at (5,1) {$\sigma'_{k+1}$};
\node at (6,1) {$\sigma'_{k+2}$};
\node at (7,1) {$\sigma'_{n-1}$};
\node at (8,1) {$\sigma'_{n}$};

\draw[<->] (4.5,1) -- (4.5,4);
\draw[<->] (5.5,1) -- (5.5,4);
\draw[<->] (7.5,1) -- (7.5,4);
\draw[<->] (8.5,1) -- (8.5,4);
\draw[<->] (6.5,1) -- (6.5,4);

\end{tikzpicture}
\end{center}
\caption{Corresponding positions when inserting 1}
\label{bij2-1}
\end{figure}

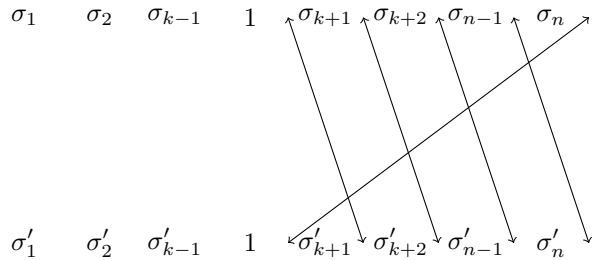
\begin{figure}[h]
\begin{center}
\begin{tikzpicture}
\node at (1,4) {$\sigma_1$};
\node at (2,4) {$\sigma_2$};
\node at (3,4) {$\sigma_{k-1}$};
\node at (4,4) {$1$};
\node at (5,4) {$\sigma_{k+1}$};
\node at (6,4) {$\sigma_{k+2}$};
\node at (7,4) {$\sigma_{n-1}$};
\node at (8,4) {$\sigma_{n}$};

\node at (1,1) {$\sigma'_1$};
\node at (2,1) {$\sigma'_2$};
\node at (3,1) {$\sigma'_{k-1}$};
\node at (4,1) {$1$};
\node at (5,1) {$\sigma'_{k+1}$};
\node at (6,1) {$\sigma'_{k+2}$};
\node at (7,1) {$\sigma'_{n-1}$};
\node at (8,1) {$\sigma'_{n}$};

\draw[<->] (4.5,1) -- (8.5,4);
\draw[<->] (5.5,1) -- (4.5,4);
\draw[<->] (6.5,1) -- (5.5,4);
\draw[<->] (7.5,1) -- (6.5,4);
\draw[<->] (8.5,1) -- (7.5,4);

\end{tikzpicture}
\end{center}
\caption{Corresponding positions when inserting 2}
\label{bij2-2}
\end{figure}

The position of $1$ in $\hat \sigma$ and $\theta(\hat \sigma)$ will be the same. Also, $pmp_{\tau_1}(\sigma)$ would increase by $1$ if and only if 2 was inserted at a non-last position, while $pmp_{\tau_2}({\sigma'})$ will increase by 1 if and only if 2 was inserted at any position but $k+1$. Thus, $pmp_{\tau_1}(\hat \sigma) = pmp_{\tau_2}(\theta(\hat \sigma))$. Also, it is not hard to see that every $\sigma \in S_n$ can be obtained by inserting 1 or 2 repeatedly in a unique way. Therefore, $\theta$ is a bijection.

As an example, we start with $\theta(12) = 12$. Say, we would like to insert 2, then the diagram looks like

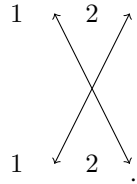
\begin{center}
\begin{tikzpicture}
\node at (1,3) {$1$};
\node at (2,3) {$2$};

\node at (1,1) {$1$};
\node at (2,1) {$2$};

\draw[<->] (1.5,1) -- (2.5,3);
\draw[<->] (2.5,1) -- (1.5,3);
\end{tikzpicture}.
\end{center}

Suppose we insert 2 to the preimage 12 at the last position, so then get $132$. According to the diagram, we should insert 2 to the image 12 at the first position, so we get $123$. Thus, $\theta(132) = 123$. Here, suppose we would like to insert 1, then the diagram looks like:

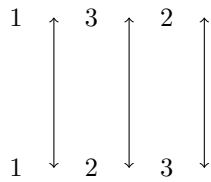
\begin{center}
\begin{tikzpicture}
\node at (1,3) {$1$};
\node at (2,3) {$3$};
\node at (3,3) {$2$};

\node at (1,1) {$1$};
\node at (2,1) {$2$};
\node at (3,1) {$3$};

\draw[<->] (2.5,1) -- (2.5,3);
\draw[<->] (3.5,1) -- (3.5,3);
\draw[<->] (1.5,1) -- (1.5,3);
\end{tikzpicture}.
\end{center}

Suppose we insert 1 to 132 at the first position, then we get 2143. According to the diagram, we insert 1 to 123 at the same position, so we get 2134. Thus, $\theta(2143) = 2134$. 

Now, suppose we want to insert 2. The diagram looks like:

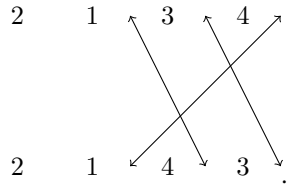
\begin{center}
\begin{tikzpicture}
\node at (1,3) {$2$};
\node at (2,3) {$1$};
\node at (3,3) {$3$};
\node at (4,3) {$4$};

\node at (1,1) {$2$};
\node at (2,1) {$1$};
\node at (3,1) {$4$};
\node at (4,1) {$3$};

\draw[<->] (3.5,1) -- (2.5,3);
\draw[<->] (4.5,1) -- (3.5,3);
\draw[<->] (2.5,1) -- (4.5,3);

\end{tikzpicture}.
\end{center}

Suppose we insert 2 to 2134 at the second position, so we get $31245$. According to the diagram, we insert 2 to $2143$ at the last position, so we get 31542. Thus, $\theta(31245) = 31542$.

Below are enumerations of $P_{n,1 \un 2 3}(x)$ for the first few $n$.

\begin{eqnarray*}
 P_{1,1 \un 2 3}(x) &=& {1} \\
 P_{2,1 \un 2 3}(x) &=&  {2} \\
 P_{3,1 \un 2 3}(x) &=&  {5 + x} \\
 P_{4,1 \un 2 3}(x) &=&  {14 + 8 x + 2 x^2}\\
 P_{5,1 \un 2 3}(x) &=&  {42 + 46 x + 26 x^2 + 6 x^3}\\
 P_{6,1 \un 2 3}(x) &=&  {132 + 232 x + 220 x^2 + 112 x^3 + 24 x^4}\\
 P_{7,1 \un 2 3}(x) &=&  {429 + 1093 x + 1527 x^2 + 1275 x^3 + 596 x^4 + 120 x^5}\\
 P_{8,1 \un 2 3}(x) &=&  {1430 + 4944 x + 9436 x^2 + 11384 x^3 + 8638 x^4 + 3768 x^5 + 720 x^6}
\end{eqnarray*}

{\bf Remark} In fact, the equivalence of $1 \un23$ and $\un213$ follows a general result from Theorem \ref{lengthn}.

\begin{section}{Equivalence classes of $S_4^*$}\label{yoyo}
There are $4! \cdot 4 = 96$ $PMP$s in $S_4^*$. However, there are at most 16 equivalent classes by Lemma \ref{sym}. All the representatives are listed in Figure \ref{allrep4}.

\begin{figure}[h]
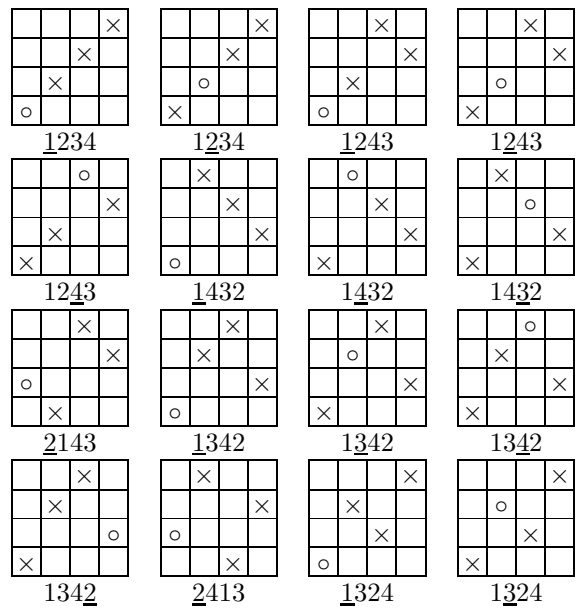

\begin{center}
\begin{tabular}{cccc}
\shortstack{$\young(\hfil \hfil \hfil \times,\hfil \hfil \times \hfil,\hfil \times \hfil \hfil,\circ \hfil \hfil \hfil )$ \\ $\un 1234 $} &
\shortstack{$\young(\hfil \hfil \hfil \times ,\hfil \hfil \times \hfil,\hfil \circ \hfil \hfil,\times \hfil \hfil \hfil )$ \\ $1\un234$} & 
\shortstack{$\young(\hfil \hfil \times \hfil  ,\hfil \hfil \hfil \times,\hfil \times \hfil \hfil,\circ \hfil \hfil \hfil )$ \\ $\un1243$} &
\shortstack{$\young(\hfil \hfil \times \hfil ,\hfil \hfil \hfil \times,\hfil \circ \hfil \hfil,\times \hfil \hfil \hfil )$ \\ $1\un243$} \\
\shortstack{$\young(\hfil \hfil \circ \hfil ,\hfil \hfil \hfil \times,\hfil \times \hfil \hfil,\times \hfil \hfil \hfil )$ \\ $12\un43$} &
\shortstack{$\young(\hfil \times \hfil \hfil  ,\hfil \hfil \times \hfil,\hfil \hfil \hfil \times,\circ \hfil \hfil \hfil )$ \\ $\un1432$} &
\shortstack{$\young(\hfil \circ \hfil \hfil  ,\hfil \hfil \times \hfil,\hfil \hfil \hfil \times,\times \hfil \hfil \hfil )$ \\ $1\un432$} &
\shortstack{$\young(\hfil \times \hfil \hfil  ,\hfil \hfil \circ \hfil,\hfil \hfil \hfil \times,\times \hfil \hfil \hfil )$ \\ $14\un32$} \\
\shortstack{$\young(\hfil \hfil \times \hfil  ,\hfil \hfil \hfil \times,\circ \hfil \hfil \hfil,\hfil \times \hfil \hfil )$ \\ $\un2143$} &
\shortstack{$\young(\hfil \hfil \times \hfil  ,\hfil \times \hfil \hfil,\hfil \hfil \hfil \times,\circ \hfil \hfil \hfil )$ \\ $\un1342$} &
\shortstack{$\young(\hfil \hfil \times \hfil  ,\hfil \circ \hfil \hfil,\hfil \hfil \hfil \times,\times \hfil \hfil \hfil )$ \\ $1\un342$} &
\shortstack{$\young(\hfil \hfil \circ \hfil  ,\hfil \times \hfil \hfil,\hfil \hfil \hfil \times,\times \hfil \hfil \hfil )$ \\ $13\un42$} \\
\shortstack{$\young(\hfil \hfil \times \hfil  ,\hfil \times \hfil \hfil,\hfil \hfil \hfil \circ,\times \hfil \hfil \hfil )$ \\ $134\un2$} &
\shortstack{$\young(\hfil \times \hfil \hfil  ,\hfil \hfil \hfil \times,\circ \hfil \hfil \hfil,\hfil \hfil \times \hfil )$ \\ $\un2413$} &
\shortstack{$\young(\hfil \hfil \hfil \times ,\hfil \times \hfil \hfil,\hfil \hfil \times \hfil,\circ \hfil \hfil \hfil )$ \\ $\un1324$} &
\shortstack{$\young(\hfil \hfil \hfil \times ,\hfil \circ \hfil \hfil,\hfil \hfil \times \hfil,\times \hfil \hfil \hfil )$ \\ $1\un324$} 
\end{tabular}
\end{center}
\label{allrep4}
\caption{Diagrams of representatives of $PMP$s of length 4 modulo geometrical equivalence}
\end{figure}

In this section, we present numerical results suggesting that there are 10 equivalent classes in $S_4^*$. We then will show that the following patterns are equivalent: $1\un234, 1\un2 43, 2\un134, 2\un1 43$.

\begin{subsection}{Numerical data}
Here, for each pattern $\tau$ shown in Figure \ref{allrep4}, we enumerate polynomials $P_{n,\tau}(x)$ for small $n$. We also group patterns together if they seem to provide the same polynomials.

{\scriptsize
\begin{subsubsection}{$\tau = \un1234, \un1243, \un1432$}

\begin{eqnarray*}
P_{1,\tau}(x) &=& 1 \\
P_{2,\tau}(x) &=& 2 \\
P_{3,\tau}(x) &=& 6 \\
P_{4,\tau}(x) &=& 23+x \\
P_{5,\tau}(x) &=& 103 + 15 x + 2 x^2 \\
P_{6,\tau}(x) &=& 513 + 158 x + 43 x^2 + 6 x^3 \\
P_{7,\tau}(x) &=& 2761 + 1466 x + 619 x^2 + 170 x^3 + 24 x^4 \\
P_{8,\tau}(x) &=& 15767 + 12864 x + 7598 x^2 + 3121 x^3 + 850 x^4 + 120 x^5 \\
\end{eqnarray*}
\end{subsubsection}

\begin{subsubsection}{$\tau = 1\un234, 1\un243, 12\un43, \un2143$}\label{4pat}
\begin{eqnarray*}
P_{1,\tau}(x) &=& 1 \\
P_{2,\tau}(x) &=& 2 \\
P_{3,\tau}(x) &=& 6 \\
P_{4,\tau}(x) &=& 23+x \\
P_{5,\tau}(x) &=& 103 + 15 x + 2 x^2 \\
P_{6,\tau}(x) &=& 513 + 157 x + 44 x^2 + 6 x^3 \\
P_{7,\tau}(x) &=& 2761 + 1439 x + 638 x^2 + 178 x^3 + 24 x^4 \\
P_{8,\tau}(x) &=& 15767 + 12420 x + 7764 x^2 + 3341 x^3 + 908 x^4 + 120 x^5
\end{eqnarray*}
\end{subsubsection}

\begin{subsubsection}{$\tau =1\un432$}
\begin{eqnarray*}
P_{1,\tau}(x) &=& 1 \\
P_{2,\tau}(x) &=& 2 \\
P_{3,\tau}(x) &=& 6 \\
P_{4,\tau}(x) &=& 23+x \\
P_{5,\tau}(x) &=& 103 + 15 x + 2 x^2 \\
P_{6,\tau}(x) &=& 513 + 157 x + 44 x^2 + 6 x^3 \\
P_{7,\tau}(x) &=& 2761 + 1438 x + 640 x^2 + 177 x^3 + 24 x^4 \\
P_{8,\tau}(x) &=& 15767 + 12393 x + 7809 x^2 + 3332 x^3 + 899 x^4 + 120 x^5
\end{eqnarray*}
\end{subsubsection}

\begin{subsubsection}{$\tau =14\un32$}
\begin{eqnarray*}
P_{1,\tau}(x) &=& 1 \\
P_{2,\tau}(x) &=& 2 \\
P_{3,\tau}(x) &=& 6 \\
P_{4,\tau}(x) &=& 23+x \\
P_{5,\tau}(x) &=& 103 + 15 x + 2 x^2 \\
P_{6,\tau}(x) &=& 513 + 156 x + 45 x^2 + 6 x^3 \\
P_{7,\tau}(x) &=& 2761 + 1415 x + 655 x^2 + 185 x^3 + 24 x^4 \\
P_{8,\tau}(x) &=& 15767 + 12058 x + 7895 x^2 + 3524 x^3 + 956 x^4 + 120 x^5
\end{eqnarray*}
\end{subsubsection}

\begin{subsubsection}{$\tau =\un1342$}
\begin{eqnarray*}
P_{1,\tau}(x) &=& 1 \\
P_{2,\tau}(x) &=& 2 \\
P_{3,\tau}(x) &=& 6 \\
P_{4,\tau}(x) &=& 23+x \\
P_{5,\tau}(x) &=& 103 + 15 x + 2 x^2 \\
P_{6,\tau}(x) &=& 512 + 160 x + 42 x^2 + 6 x^3 \\
P_{7,\tau}(x) &=& 2740 + 1500 x + 614 x^2 + 162 x^3 + 24 x^4 \\
P_{8,\tau}(x) &=& 15485 + 13207 x + 7700 x^2 + 3016 x^3 + 792 x^4 + 120 x^5
\end{eqnarray*}
\end{subsubsection}

\begin{subsubsection}{$\tau =1\un342$}
\begin{eqnarray*}
P_{1,\tau}(x) &=& 1 \\
P_{2,\tau}(x) &=& 2 \\
P_{3,\tau}(x) &=& 6 \\
P_{4,\tau}(x) &=& 23+x \\
P_{5,\tau}(x) &=& 103 + 15 x + 2 x^2 \\
P_{6,\tau}(x) &=& 512 + 158 x + 44 x^2 + 6 x^3 \\
P_{7,\tau}(x) &=& 2740 + 1451 x + 646 x^2 + 179 x^3 + 24 x^4 \\
P_{8,\tau}(x) &=& 15485 + 12455 x + 7912 x^2 + 3427 x^3 + 921 x^4 + 120 x^5
\end{eqnarray*}
\end{subsubsection}

\begin{subsubsection}{$\tau =13\un42, \un2413$}
\begin{eqnarray*}
P_{1,\tau}(x) &=& 1 \\
P_{2,\tau}(x) &=& 2 \\
P_{3,\tau}(x) &=& 6 \\
P_{4,\tau}(x) &=& 23+x \\
P_{5,\tau}(x) &=& 103 + 15 x + 2 x^2 \\
P_{6,\tau}(x) &=& 512 + 158 x + 44 x^2 + 6 x^3 \\
P_{7,\tau}(x) &=& 2740 + 1454 x + 644 x^2 + 178 x^3 + 24 x^4 \\
P_{8,\tau}(x) &=& 15485 + 12533 x + 7897 x^2 + 3377 x^3 + 908 x^4 + 120 x^5
\end{eqnarray*}
\end{subsubsection}

\begin{subsubsection}{$\tau =134\un2$}
\begin{eqnarray*}
P_{1,\tau}(x) &=& 1 \\
P_{2,\tau}(x) &=& 2 \\
P_{3,\tau}(x) &=& 6 \\
P_{4,\tau}(x) &=& 23+x \\
P_{5,\tau}(x) &=& 103 + 15 x + 2 x^2 \\
P_{6,\tau}(x) &=& 512 + 159 x + 43 x^2 + 6 x^3 \\
P_{7,\tau}(x) &=& 2740 + 1475 x + 629 x^2 + 172 x^3 + 24 x^4 \\
P_{8,\tau}(x) &=& 15485 + 12817 x + 7781 x^2 + 3244 x^3 + 873 x^4 + 120 x^5
\end{eqnarray*}
\end{subsubsection}

\begin{subsubsection}{$\tau =\un1324$}
\begin{eqnarray*}
P_{1,\tau}(x) &=& 1 \\
P_{2,\tau}(x) &=& 2 \\
P_{3,\tau}(x) &=& 6 \\
P_{4,\tau}(x) &=& 23+x \\
P_{5,\tau}(x) &=& 103 + 15 x + 2 x^2 \\
P_{6,\tau}(x) &=& 513 + 158 x + 43 x^2 + 6 x^3 \\
P_{7,\tau}(x) &=& 2762 + 1464 x + 620 x^2 + 170 x^3 + 24 x^4 \\
P_{8,\tau}(x) &=& 15793 + 12820 x + 7608 x^2 + 3129 x^3 + 850 x^4 + 120 x^5
\end{eqnarray*}
\end{subsubsection}

\begin{subsubsection}{$\tau =1\un324$}
\begin{eqnarray*}
P_{1,\tau}(x) &=& 1 \\
P_{2,\tau}(x) &=& 2 \\
P_{3,\tau}(x) &=& 6 \\
P_{4,\tau}(x) &=& 23+x \\
P_{5,\tau}(x) &=& 103 + 15 x + 2 x^2 \\
P_{6,\tau}(x) &=& 513 + 156 x + 45 x^2 + 6 x^3 \\
P_{7,\tau}(x) &=& 2762 + 1414 x + 654 x^2 + 186 x^3 + 24 x^4 \\
P_{8,\tau}(x) &=& 15793 + 12041 x + 7861 x^2 + 3539 x^3 + 966 x^4 + 120 x^5
\end{eqnarray*}

\end{subsubsection}
}

\end{subsection}

\begin{subsection}{Equivalence of $1\un234, 1\un243, \un2134$ and $\un2 1 4 3$}

Here, we prove an equivalence of four patterns: $1\un234, 1\un243, \un2134$ and $\un2 1 4 3$. Note that, they are equivalent to four patterns appearing in Section \ref{4pat}, since $\un 2134$ and $12\un43$ are equivalent.

\begin{center}
\young(\hfill \hfil \hfill \times
,\hfill \hfill \times \hfil
,\circ \hfill \hfill \hfil
,\hfil \times \hfil \hfil)
\young(\hfill \hfil \circ \hfill
,\hfill \hfill \hfil \times
,\hfil \times \hfill \hfill
,\times \hfil \hfil \hfil)

Diagrams of $\un2134$ and $12 \un43$
\end{center}

The equivalence of $1\un234$ and $\un2134$ as well as the equivalence of $1\un243$ and $\un2143$ follow a more general Theorem \ref{lengthn}. Thus, in this section, we only prove the equivalence of $1\un234$ and $1\un243$.

\begin{theorem}\label{len4}
Two $PMP$s $1\un234$ and $1\un243$ are $PMP$-Wilf-equivalent.
\end{theorem}

For our convenience, let $\tau_1 = 1 \un234$. Let $P_{n,\tau_1,s}(x) = \sum_{\sigma} x^{pmp_{\tau_1}(\sigma)}$, where the sum is over all $\sigma \in S_n$ with the last ascent of $\sigma$ at position $s$. Let $P_{n,\tau_1,s,t}(x) = \sum_{\sigma}x^{pmp_{\tau_1}(\sigma)}$, where the sum is over all $\sigma \in S_n$ with the last ascent of $\sigma$ at position $s$ and 1 at position $t$. All feasible values of $s$ are $1,2,\ldots, n-1$. All feasible values of $t$ are $1,2,\ldots,s,n$ for $s \neq n-1$, and $1,2,\ldots,n-1$ for $s=n-1$. Thus, 
$$P_{n,\tau_1}(x) = 1+ \sum_{s=1}^{n-1} P_{n,\tau_1,s}(x)$$
and also, 
\begin{equation}\label{start_t1} P_{n,\tau_1,s}(x) = \left ( \sum_{t=1}^{s} P_{n,\tau_1,s,t}(x) \right ) + P_{n,\tau_1,s,n}(x).
\end{equation}

Note that the second term vanishes as $s = n-1$. Here, we derive a recursive formula for $P_{n,\tau_1,s,t}(x)$.

\begin{lemma}\label{recur_t1} $P_{n,\tau_1,s,t}(x)$ satisfies the following recursive formulas:
$$P_{n,\tau_1,s,t}(x) = \sum_{\ell=1}^{t-1} P_{n-1,\tau_1,s-1,\ell}(x) + (s-1-t)xP_{n-1,\tau_1,s-1,t}(x) + \sum_{j=t+1}^{s}P_{n-1,\tau_1,j,t}(x) + P_{n-1,\tau_1,t,t}(x) $$
for $1\le t<s \le n-1, $

$$P_{n,\tau_1,s,s}(x) = 1+ \sum_{i=1}^{s-1}P_{n-1,\tau_1,i}(x)$$
for $1 \le s \le n-1$, and

$$P_{n,\tau_1,s,n}(x) = P_{n-1,\tau_1,s}(x)$$
for $1 \le s < n-1$.

By convention, Let $P_{n,\tau_1,s,t}(x) = 0$ for infeasible value of $s,t$.
\end{lemma}

\begin{proof}
First, consider the formula for $P_{n,\tau_1,s,t}(x)$ when $1\le t <s \le n-1.$ Let $\ell$ be the position of $2$ in $\sigma$, where $\sigma\in S_n$ with the last ascent of $\sigma$ at position $s$ and 1 at position $t$. We have 4 cases:

\begin{enumerate}
\item $1 \le \ell \le t-1$. In this case, we can remove $1$ from $\sigma$ without effecting a $\tau_1$-match at other positions. That is because $\sigma$ does not have a $\tau_1$-match at the position of $2$, and if $\sigma$ has a $\tau_1$-match at other positions and matching involves 1, it can still have a $\tau_1$-match by using 2 instead of 1.

After removing and reducing, the position of 1 is $\ell$, and the position of the last ascent is $s-1$. Thus, this case contributes $\sum_{\ell=1}^{t-1}P_{n-1,\tau_1,s-1,\ell}(x)$.

\item $t+1 \le \ell \le s-1$. In this case, $\sigma$ has a $\tau_1$-match at position $\ell$ since $1 ~ 2 ~\sigma_s ~ \sigma_{s+1}$ form the pattern $\tau_1$. By the same reasoning as the first case, we can remove 2 without effecting $\tau_1$-match at any other positions.

After removing and reducing, the position of $1$ stays at $t$ and the position of the last ascent is $s-1$. Thus, this case contributes $x(s-1-t)P_{n-1,\tau_1,s-1,t}(x)$.

\item $\ell = s$. In this case, we can remove 2 without effecting a $\tau_1$-match at other positions.

After removing and reducing, the position of $1$ stays at $t$ and the last ascent appears at some position between $t$ and $s-1$. Thus, this case contributes $\sum_{j=t}^{s-1}P_{n-1,\tau_1,j,t}(x)$.

\item $\ell = n$. In this case, we can remove 2 without effecting a $\tau_1$-match at other positions. After removing and reducing, the position of 1 stays at $t$ and the position of the last ascent stays at $s$. Thus, this case contributes $P_{n-1,\tau_1,s,t}(x)$.
\end{enumerate}

In total, we have

\begin{eqnarray*}
P_{n,\tau_1,s,t}(x) &=& \sum_{\ell=1}^{t-1} P_{n-1,\tau_1,s-1,\ell}(x) + (s-1-t)xP_{n-1,\tau_1,s-1,t}(x) + \sum_{j=t}^{s-1} P_{n-1,\tau_1,j,t}(x) + P_{n-1,\tau_1,s,t}(x) \\
&=& \sum_{\ell=1}^{t-1} P_{n-1,\tau_1,s-1,\ell}(x) + (s-1-t)xP_{n-1,\tau_1,s-1,t}(x) + \sum_{j=t+1}^{s} P_{n-1,\tau_1,j,t}(x) + P_{n-1,\tau_1,t,t}(x).
\end{eqnarray*}

Note that, the first case vanishes if $t=1$, but the formula is consistent as it contributes an empty summation. The second case vanishes when $t = s-1$, but the formula is still correct as $P_{n-1,\tau_1,s-1,s}(x) = 0$. The last case vanishes when $s = n-1$, but the formula is correct as $P_{n-1,\tau_1,n-1,t}(x)=0$. Thus, we prove the first formula.

Here, we derive a recursive formula for $P_{n,\tau_1,s,s}(x)$ when $1\le s\le n-1$. Given any $\sigma\in S_n$ such that the position of 1 and the position of the last ascent is $s$. Removing 1 does not effect a $\tau_1$-match at other positions. After removing and reducing, the permutation either has no ascent, or the last ascent is at some position between 1 and $s-1$. Thus, we have

$$P_{n,\tau_1,s,s}(x) = 1 + \sum_{i=1}^{s-1}P_{n-1,\tau_1,i}(x).$$

Lastly, we derive a formula for $P_{n,\tau_1,s,n}(x)$ when $1 \le s < n-1$. Given any $\sigma$ such that the position of the last ascent is $s$ and 1 is at the last position. Removing 1 will not effect $\tau_1$-match at other positions. The position of the last ascent is still $s$. Thus, we have

$$P_{n,\tau_1,s,n}(x) = P_{n-1,\tau_1,s}(x).$$

\end{proof}

We shall prove similar result for $1\un243$. Let $\tau_2 = 1 \un2 43$. Let $P_{n,\tau_2,s}(x) = \sum_{\sigma} x^{pmp_{\tau_2}(\sigma)}$, where the sum is over all $\sigma \in S_n$ with the {\em last descent} of $\sigma$ at position $s$. Let $P_{n,\tau_2,s,t}(x) = \sum_{\sigma}x^{pmp_{\tau_2}(\sigma)}$, where the sum is over all $\sigma \in S_n$ with the last descent of $\sigma$ at position $s$ and 1 at position $t$. All feasible values of $s$ are $1,2,\ldots, n-1$. All feasible values of $t$ are $t = 1,2, \ldots, s-1$ and $s+1$. Note that, when $s=1$, the only possible value for $t$ is 2. Thus, we have

$$P_{n,\tau_2}(x) = 1+ \sum_{s=1}^{n-1} P_{n,\tau_2,s}(x).$$

Also, 

\begin{equation}\label{start_t2}
P_{n,\tau_2,s}(x) = \left ( \sum_{t=1}^{s-1} P_{n,\tau_2,s,t}(x) \right ) + P_{n,\tau_2,s,s+1}(x).
\end{equation}

Here, we derive a recursive formula for $P_{n,\tau_2,s}(x)$.

\begin{lemma}\label{recur_t2}  $P_{n,\tau_2,s,t}(x)$ satisfies the following recursive formulas:
$$P_{n,\tau_2,s,t}(x) = \sum_{\ell=1}^{t-1} P_{n-1,\tau_2,s-1,\ell}(x) + (s-1-t)xP_{n-1,\tau_2,s-1,t}(x) + \sum_{j=t+1}^{s}P_{n-1,\tau_2,j,t}(x) + P_{n-1,\tau_2,t-1,t}(x).$$
for $1\le t<s \le n-1, $

$$P_{n,\tau_2,s,s+1}(x) = 1+ \sum_{i=1}^{s}P_{n-1,\tau_2,i}(x)$$
for $1 \le s \le n-1$, and

By convention, let $P_{n,\tau_2,s,t}(x) = 0$ for infeasible value of $s,t$.
\end{lemma}

\begin{proof}
First, consider the formula for $P_{n,\tau_2,s,t}(x)$ when $1 \le t < s \le n-1$. Let $\ell$ be the position of 2 in $\sigma$, where $\sigma \in S_n$ with the last descent of $\sigma$ at position $s$ and 1 at position $t$. we have 4 cases:

\begin{enumerate}
\item $1 \le \ell \le t-1$. In this case, $\sigma$ does not have a $\tau_2$-match at position $\ell$. We can remove 1 from $\sigma$ without effecting $\tau_2$-match at other positions. That is because if $\sigma$ has a $\tau_2$-match at other positions and the matching involves 1, it still has a $\tau_2$-match by using $2$ instead of $1$.

After removing and reducing, the position of $1$ is $\ell$ and the position of the last descent is $s-1$. Thus, this case contributes $\sum_{\ell=1}^{t-1}P_{n-1,\tau_2,s-1,\ell}(x)$.

\item $t+1 \le \ell \le s-1$. In this case, $\sigma$ has a $\tau_2$-match at position $\ell$ since $1 ~ 2 ~\sigma_s ~ \sigma_{s+1}$ form the pattern $\tau_2$. By the same reasoning as the first case, we can remove 2 without effecting $\tau_2$-match at any other positions.

After removing and reducing, the position of $1$ stays at $t$ and the position of the last descent is $s-1$. Thus, this case contributes $x(s-1-t)P_{n-1,\tau_1,s-1,t}(x)$.

\item $\ell = s$. In this case, we can remove 2 without effecting $\tau_2$-match at other positions.

After removing and reducing, the position of $1$ stays at $t$ and the last descent appears at some position between $t+1$ and $s$, or at the position $t-1$ (right before 1). Thus, this case contributes $\sum_{j=t+1}^{s}P_{n-1,\tau_2,j,t}(x) + P_{n-1,\tau_2,t-1,t}(x)$.

\end{enumerate}

In total, we have

$$P_{n,\tau_2,s,t}(x) = \sum_{\ell=1}^{t-1} P_{n-1,\tau_2,s-1,\ell}(x) + (s-1-t)xP_{n-1,\tau_2,s-1,t}(x) + \sum_{j=t+1}^{s}P_{n-1,\tau_2,j,t}(x) + P_{n-1,\tau_2,t-1,t}(x).$$

Note that, the first case vanishes if $t=1$, but the formula is consistent as it contributes an empty summation. The second case vanishes when $t = s-1$, but the formula is still correct as $P_{n-1,\tau_2,s-1,s}(x) = 0$. The last case vanishes when $s = n-1$, but the formula is correct as $P_{n-1,\tau_2,n-1,t}(x)=0$. Thus, we prove the first formula.

Next, we derive the recursive formula for $P_{n,\tau_2,s,s+1}(x)$ for $1 \le s \le n-1$. Given any $\sigma \in S_n$ such that the position of the last descent is $s$ and the position of 1 is $s+1$. Removing 1 will not effect $\tau_2$-matching at other positions. After removing and reducing, the remaining permutation either has no descent, or the last descent is at some position between 1 and $s$. Thus, we have

$$P_{n,\tau_2,s,s+1}(x) = 1+ \sum_{i=1}^{s}P_{n-1,\tau_2,i}(x).$$

Thus, we prove the second formula.

\end{proof}

Here, we prove equality of $P_{n,\tau_1}(x)$ and $P_{n,\tau_2}(x)$. Note that since $P_{n,\tau_1}(x) = 1+\sum_{s=1}^{n-1} P_{n,\tau_1,s}(x)$ and $P_{n,\tau_2}(x) = 1+\sum_{s=1}^{n-1} P_{n,\tau_2,s}(x)$, it is enough to show that $P_{n,\tau_1,s}(x) = P_{n,\tau_2,s}(x)$ for $1\le s \le n-1$.

\begin{lemma}
$P_{n,\tau_1,s}(x) = P_{n,\tau_2,s}(x)$ for $1\le s \le n-1$.

\end{lemma}

\begin{proof}

First of all, $P_{n,\tau_1,t,t}(x) = P_{n,\tau_2,t-1,t}(x)$ by Lemma \ref{recur_t1} and \ref{recur_t2}. As a result, $P_{n,\tau_1,s,t}(x)$ and $P_{n,\tau_2,s,t}(x)$ have the same recursive formula for $1\le t <s$. Then, we proceed to prove the lemma as follows.

\begin{align*}
P_{n,\tau_1,s}(x) &= \sum_{t=1}^s P_{n,\tau_1,s,t}(x) + P_{n,\tau_1,s,n}(x) &\text{By } (\ref{start_t1})\\
&= \sum_{t=1}^{s-1} P_{n,\tau_1,s,t}(x) + P_{n,\tau_1,s,s}(x) + P_{n,\tau_1,s,n}(x) \\
&= \sum_{t=1}^{s-1} P_{n,\tau_1,s,t}(x) + \left ( 1+\sum_{\ell=1}^{s-1} P_{n-1,\tau_1,\ell}(x) \right ) + P_{n-1,\tau_1,s}(x) &\text{By Lemma } \ref{recur_t1}\\
&= \sum_{t=1}^{s-1} P_{n,\tau_1,s,t}(x) + 1+\sum_{i=1}^{s} P_{n-1,\tau_1,i}(x) 
\end{align*}

Here, notice that $P_{n,\tau_1,s,t}(x) = P_{n,\tau_2,s,t}(x)$ as we discussed above, and $P_{n-1,\tau_1,i}(x) = P_{n-1,\tau_2,i}(x)$ by inductive hypothesis. Thus, we have
\begin{align*}
P_{n,\tau_1,s}(x) &= \sum_{t=1}^{s-1} P_{n,\tau_2,s,t}(x) + 1+\sum_{i=1}^{s} P_{n-1,\tau_2,i}(x) \\
&= \sum_{t=1}^{s-1} P_{n,\tau_2,s,t}(x) + P_{n,\tau_2,s,s+1}(x) &\text{By Lemma } \ref{recur_t2}\\
&= P_{n,\tau_2,s}(x) &\text{By } (\ref{start_t2})
\end{align*}

Therefore, $P_{n,\tau_1,s}(x) = P_{n,\tau_2,s}(x)$ for all $1 \le s \le n-1$.

\end{proof}

As a consequence, $P_{n,\tau_1}(x) = P_{n,\tau_2}(x)$, and so $\tau_1$ and $\tau_2$ are $PMP$-Wilf-equivalent. Hence, we prove Theorem \ref{len4}.

We end this section here by introducing collections of patterns that are likely to be equivalent to one another, but we do not yet have a proof. According to numerical values generated, we guess that $\un1 2 3 4, \un 1 4 3 2$ and $\un 1 2 4 3$ are all equivalent to one another. Also, $1 3 \un 4 2$ and $\un 2 4 1 3$ are likely to be equivalent. Diagrams of mentioned positional marked patterns are shown in Figure \ref{conj1} and \ref{conj2}.

\begin{figure}[h]
\begin{center}
\shortstack{$\young(\hfil \hfil \hfil \times,\hfil \hfil \times \hfil,\hfil \times \hfil \hfil,\circ \hfil \hfil \hfil )$ \\ $\un 1234 $} 
\shortstack{$\young(\hfil \times \hfil \hfil  ,\hfil \hfil \times \hfil,\hfil \hfil \hfil \times,\circ \hfil \hfil \hfil )$ \\ $\un1432$}
\shortstack{$\young(\hfil \hfil \times \hfil  ,\hfil \hfil \hfil \times,\hfil \times \hfil \hfil,\circ \hfil \hfil \hfil )$ \\ $\un1243$}
\end{center}
\caption{Patterns $\un 1234, \un1432 $ and $\un1243$}
\label{conj1}
\end{figure}

\begin{figure}[h]
\begin{center}
\shortstack{$\young(\hfil \times \hfil \hfil  ,\hfil \hfil \hfil \times,\circ \hfil \hfil \hfil,\hfil \hfil \times \hfil )$ \\ $\un2413$}
\shortstack{$\young(\hfil \hfil \circ \hfil  ,\hfil \times \hfil \hfil,\hfil \hfil \hfil \times,\times \hfil \hfil \hfil )$ \\ $13\un42$}\end{center}
\caption{Patterns $\un2413$ and $13\un42$}
\label{conj2}
\end{figure}

\end{subsection}

\end{section}

\begin{section}{Patterns of arbitrary length}\label{lenn}

In this section, we restate and prove an equivalence of a collection positional marked patterns of arbitrary length mentioned in Section \ref{intro}.

\lengthn*

We will prove that $P_1$ and $P_2$ are $PMP$-Wilf-equivalent. The equivalence of $P_1$ and $P_3$ follows from the fact that $P_1$ and $P_2$ are equivalent. Here, we apply the technique introduced in \cite{BW} to prove the equivalence of $P_1$ and $P_2$. First, we need to give several definitions. Let $\tau = red(p_1p_2 \ldots p_{\ell-2})$.

\begin{definition}
Given $\pi \in S_n$. Consider the diagram of $\pi$. For each cell $(i,j)$ in the diagram of $\pi$, the cell is {\em dominant} if there is an occurrence of $\tau$ in the diagram of $\pi$ when only considering row $i+1,i+2,\ldots,n$ and columns $j+1, j+2,\ldots,n$. A cell is {\em non-dominant} if it is not dominant.
\end{definition}

It is clear that given any dominant cell, every cell to the left and below the dominant cell is also dominant. Thus, the collection of dominant cells form a Ferrers board.

Let $ND(\pi) = \set{(i,j) \in [n] \times [n] \hspace{.2cm} | \hspace{.2cm} \text{cell } (i,j) \text{ is non-dominant and contains } \times}$. Note that, if a cell is dominant, then one could find an occurrence of $\tau$ above and to the right of the cell such that all $\times$ involved are in non-dominant cell. If not, then every occurrence of $\tau$ above and to the right of the cell $(i,j)$ contains $\times$ in a dominant cell. Pick a copy $T_1$ of $\tau$ in which a dominant cell $(i',j')$ containing $\times$ is the rightmost among all dominant cells containing $\times$ in all copies of $\tau$ above and to the right of $(i,j)$. Since cell $(i',j')$ is also dominant, one can find a copy $T_2$ of $\tau$ above and to the right of $(i',j')$, which is also above and to the right of cell $(i,j)$. Thus, this copy will also contain a dominant cell containing $\times$, contradicting to the fact that the $T_1$ contains the rightmost dominant cell.

Thus, if $ND(\pi)$ is known, one can recover the set of dominant cells in the diagram of $\pi$ completely. Given any $Q \subseteq [n] \times [n]$, let $S^Q_n = \set{\pi \in S_n | ND(\pi) = Q}$. We shall prove that

$$\sum_{\pi \in S^Q_n} x^{pmp_{P_1}(\pi)} = \sum_{\pi \in S^Q_n} x^{pmp_{P_2}(\pi)}.$$

Here, we analyse the set $S_n^Q$. First, we only need to consider those $Q$ such that $S_n^Q \neq \emptyset$. Given any such $Q$, we can obtain an element in $S_n^Q$ by filling $\times$ in the dominant part of the diagram until every row and column contains precisely one $\times$. This gives all elements in $S_n^Q$ since filling $\times$ in the dominant part does not alter whether a cell is dominant or non-dominant.

To fill $\times$ in dominant cells, we start with all dominant cells and eliminate all rows and columns that already contain $\times$ from the set $Q$. The remaining cells form a Ferrers board. Let $\lambda(Q)$ denote the shape of the Ferrers board obtained from the process above.

\begin{definition}
Given any Ferrers board of shape $\lambda = (\la_1,\la_2,\ldots,\la_k)$ with $\la_1 \ge \la_2 \ge \ldots \ge \la_k$, {\em a filling of $\lambda$} is an assignment of $\times$ in the Ferrers board of shape $\lambda$ such that every row and columns contain precisely one $\times$.
\end{definition}

In order for $\lambda$ to have a filling, the number of rows of $\lambda$ must be the same as the number of columns in $\lambda$. More specifically, if $\lambda = (\lambda_1,\lambda_2, \ldots, \lambda_k)$, $\lambda$ has a filling if and only if $(k-i+1) \le \lambda_i \le k$ for all $k$. Let $S_\lambda$ denote the set of fillings of $\lambda$.

Given any $\pi \in S_\lambda$, and $p \in S_k$, we say that $\pi$ contains $p$ if $\pi$ contains a subdiagram of $p$. Note that, in order for $\pi$ to contain $p$, the {\em entire} diagram of $p$ has to present as a subdiagram of $\pi$ including cells not containing $\times$. For example, the filling below contains $p=12$ but does not contain $q = 21$.

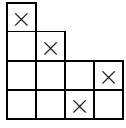
\begin{center}\young(\times,\hfil \times,\hfil \hfil \hfil \times,\hfil \hfil \times \hfil)\end{center}

Here, we define positional marked pattern on $S_\lambda$ similar to $S_n$. Given any $\pi \in S_\lambda$ and $p \in S_k^*$, we say that $\pi$ has a $p$-match at position $t$ if when replacing $\times$ in the $t$-th column of $\pi$ of by $\circ$, then it contains a subdiagram $p$. For example, the diagram above has a $\un 1 2$-match at position 3 as shown below.

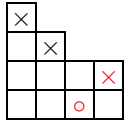
\begin{center}\young(\times,\hfil \times,\hfil \hfil \hfil \redx,\hfil \hfil \redo \hfil)\end{center}

Let $pmp_p(\pi)$ denote the number of positions $t$ such that $\pi$ has a $p$-match at position $t$. Here, we state a lemma which will be our main tool to prove Theorem \ref{lengthn}.

\begin{lemma} \label{bite} Given any $Q\subseteq [n]\times [n]$ such that $S_n^Q \neq \emptyset$. Then

$$\sum_{\pi \in S_n^Q} x^{pmp_{P_1}(\pi)} = \sum_{\pi \in S_\lambda(Q)} x^{pmp_{1\un2} (\pi)}$$

and

$$\sum_{\pi \in S_n^Q} x^{pmp_{P_2}(\pi)} = \sum_{\pi \in S_\lambda(Q)} x^{pmp_{\un2 1} (\pi)}.$$
\end{lemma}

\begin{proof}
Note that there is a natural bijection $\phi: S_{\lambda(Q)} \rightarrow S_n^Q$ which maps any filling $\pi \in \la(Q)$ to a permutation with non-dominant part corresponding to $Q$ and dominant part having the same filling as $\pi$. We will show that $\pi$ has a $1\un2$-match at a position if and only if $\phi(\pi)$ has a $P_1$-match at the corresponding position according to the map $\phi$

The converse is obvious. For the forward direction, suppose $\pi \in S_{\lambda(Q)}$ has a $1\un2$-match at a certain position. Thus, the corresponding position in $\phi(\pi)$ has a $1\un2$-match with every cells involved are dominant. Thus, we can find a copy of $\tau$ above and to the right of every cells involved in $1\un2$-match. Therefore, the $1\un2$-match together with $\tau$ makes $\phi(\pi)$ $P_1$-match at the position corresponding to the $1\un2$-match in $\pi$. 

Therefore, $\phi$ is a bijection between $S_{\la(Q)}$ and $S_n^Q$ such that $pmp_{1\un2}(\pi) = pmp_{P_1}(\phi(\pi))$. Thus, it proves the first equality of the lemma. The second equation can be proved with the exact same reasoning.
\end{proof}

As an example, let $n=9$, $\tau = 12$, (so $P_1 = 1\un234$ and $P_2 = \un2 1 3 4$), and $Q$ is as below

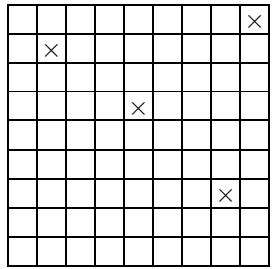
\begin{center}
\young(\hfil \hfil \hfil \hfil \hfil \hfil \hfil \hfil \times,\hfil \times \hfil \hfil \hfil \hfil \hfil \hfil \hfil,\hfil \hfil \hfil \hfil \hfil \hfil \hfil \hfil \hfil,\hfil \hfil \hfil \hfil \times \hfil \hfil \hfil \hfil,\hfil \hfil \hfil \hfil \hfil \hfil \hfil \hfil \hfil,\hfil \hfil \hfil \hfil \hfil \hfil \hfil \hfil \hfil,\hfil \hfil \hfil \hfil \hfil \hfil \hfil \times \hfil,\hfil \hfil \hfil \hfil \hfil \hfil 
\hfil \hfil \hfil,\hfil \hfil \hfil \hfil \hfil \hfil 
\hfil \hfil \hfil) .
\end{center}
 
With $Q$, we recover dominant and non-dominant cells. We fill $\bullet$ in non-dominant cells as below

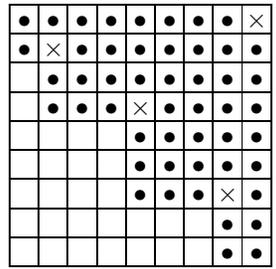
\begin{center}
\young(\bullet \bullet \bullet \bullet \bullet \bullet \bullet \bullet \times,\bullet \times \bullet \bullet \bullet \bullet \bullet \bullet \bullet,\hfil \bullet \bullet \bullet \bullet \bullet \bullet \bullet \bullet,\hfil \bullet \bullet \bullet \times \bullet \bullet \bullet \bullet,\hfil \hfil \hfil \hfil \bullet \bullet \bullet \bullet \bullet,\hfil \hfil \hfil \hfil \bullet \bullet \bullet \bullet \bullet,\hfil \hfil \hfil \hfil \bullet \bullet \bullet \times \bullet,\hfil \hfil \hfil \hfil \hfil \hfil \hfil \bullet \bullet,\hfil \hfil \hfil \hfil \hfil \hfil \hfil \bullet \bullet) .
\end{center}

We want to fill $\times$ into the diagram so that every row and column contains precisely one $\times$, thus we eliminate all empty cells that are in same rows or columns with cell containing $\times$. We fill $\bullet$ in such cells,

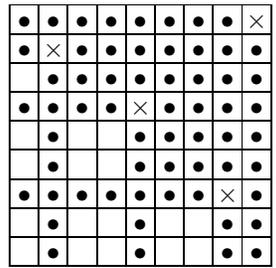
\begin{center}
\young(\bullet \bullet \bullet \bullet \bullet \bullet \bullet \bullet \times
,\bullet \times \bullet \bullet \bullet \bullet \bullet \bullet \bullet
,\hfil \bullet \bullet \bullet \bullet \bullet \bullet \bullet \bullet
,\bullet \bullet \bullet \bullet \times \bullet \bullet \bullet \bullet
,\hfil \bullet \hfil \hfil \bullet \bullet \bullet \bullet \bullet
,\hfil \bullet \hfil \hfil \bullet \bullet \bullet \bullet \bullet
,\bullet \bullet \bullet \bullet \bullet \bullet \bullet \times \bullet
,\hfil \bullet \hfil \hfil \bullet \hfil \hfil \bullet \bullet
,\hfil \bullet \hfil \hfil \bullet \hfil \hfil \bullet \bullet) .
\end{center}

Thus, the remaining cells form a Ferrers board $\lambda(Q) = (5,5,3,3,1)$, as shown below:

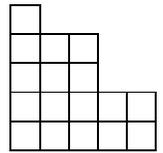
\begin{center}
\young(\hfil
,\hfil\hfil\hfil
,\hfil\hfil\hfil
,\hfil\hfil\hfil\hfil\hfil
,\hfil\hfil\hfil\hfil\hfil) .
\end{center}

Then, to define the one-to-one corresponding between $S_9^Q$ and $S_{\la(Q)}$, we fill cells in $\lambda(Q)$ the same way we fill available cells in a diagram $Q$. For example, below is an example of the correspondence:

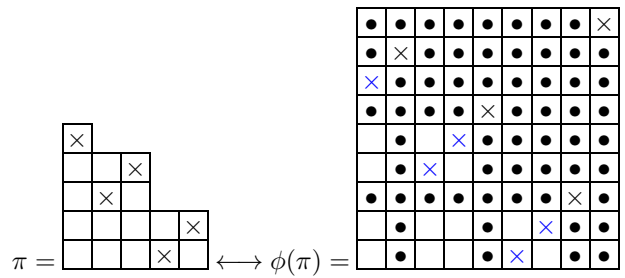
\begin{center}
$\pi =$
\young(\times
,\hfil\hfil\times
,\hfil\times\hfil
,\hfil\hfil\hfil\hfil\times
,\hfil\hfil\hfil\times\hfil)
$\longleftrightarrow$
$\phi(\pi) =$
\young(\bullet \bullet \bullet \bullet \bullet \bullet \bullet \bullet \times
,\bullet \times \bullet \bullet \bullet \bullet \bullet \bullet \bullet
,\bluex \bullet \bullet \bullet \bullet \bullet \bullet \bullet \bullet
,\bullet \bullet \bullet \bullet \times \bullet \bullet \bullet \bullet
,\hfil \bullet \hfil \bluex \bullet \bullet \bullet \bullet \bullet
,\hfil \bullet \bluex \hfil \bullet \bullet \bullet \bullet \bullet
,\bullet \bullet \bullet \bullet \bullet \bullet \bullet \times \bullet
,\hfil \bullet \hfil \hfil \bullet \hfil \bluex \bullet \bullet
,\hfil \bullet \hfil \hfil \bullet \bluex \hfil \bullet \bullet) .
\end{center}

With this map, $1\un2$-matching at position 3 of $\pi$ corresponds to $1 \un 2 3 4$-matching at position 4 of $\phi(\pi)$,

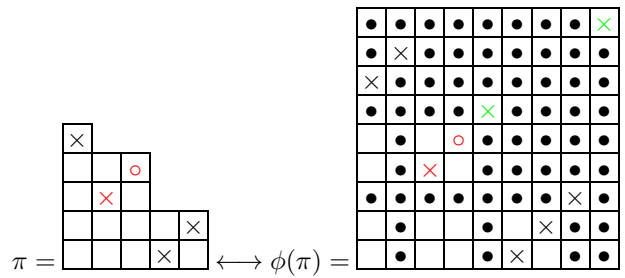
\begin{center}
$\pi =$
\young(\times
,\hfil\hfil\redo
,\hfil\redx\hfil
,\hfil\hfil\hfil\hfil\times
,\hfil\hfil\hfil\times\hfil)
$\longleftrightarrow$
$\phi(\pi) =$
\young(\bullet \bullet \bullet \bullet \bullet \bullet \bullet \bullet \greenx
,\bullet \times \bullet \bullet \bullet \bullet \bullet \bullet \bullet
,\times \bullet \bullet \bullet \bullet \bullet \bullet \bullet \bullet
,\bullet \bullet \bullet \bullet \greenx \bullet \bullet \bullet \bullet
,\hfil \bullet \hfil \redo\bullet \bullet \bullet \bullet \bullet
,\hfil \bullet \redx \hfil \bullet \bullet \bullet \bullet \bullet
,\bullet \bullet \bullet \bullet \bullet \bullet \bullet \times \bullet
,\hfil \bullet \hfil \hfil \bullet \hfil \times \bullet \bullet
,\hfil \bullet \hfil \hfil \bullet \times \hfil \bullet \bullet) .
\end{center}

Here, in order to prove theorem \ref{lengthn}, we only need to prove that, given any Ferrers board $\lambda$, $\sum_{\pi \in S_\la} x^{pmp_{1\un2}(\pi)} = \sum_{\pi \in S_\la} x^{pmp_{\un21}(\pi)}$. In fact, we find a formula for both polynomials.

\begin{lemma} \label{equivla1221}
Let $\la = (\la_1,\la_2,\ldots,\la_k)$ be a Ferrers board such that $(k-i+1) \le \la_i  \le k$ for all $1\le i \le k$. Then,

$$\sum_{\pi \in S_\la} x^{pmp_{1\un2}(\pi)}  = \sum_{\pi \in S_\la} x^{pmp_{\un21}(\pi)} = \prod_{i=1}^k (1+ (\la_i-(k-i+1))x).$$
\end{lemma}

\begin{proof}
Given any $\la$ satisfying $(k-i+1) \le \la_i  \le k$ for all $1\le i \le k$, let $\bar \la$ be a Ferrers board obtained from $\la$ by removing the top most row and left most column. That is $\bar \la = (\la_1-1,\la_2-1,\ldots, \la_{k-1}-1)$. Then, to prove the formula, we only need to prove that 

$$\sum_{\pi \in S_\la} x^{pmp_{1\un2}(\pi)}  = (1+ (\la_k-1)x)\sum_{\pi \in S_{\bar \la}} x^{pmp_{1\un2}(\pi)}.$$

To prove the equation above, we consider the topmost row in $\la$. The filling has a $1\un2$-match at the position of $\times$ in the topmost row if and only if $\times$ is not in the rightmost possible cell. So, there are $\la_k-1$ positions to fill $\times$ so that it has a $1\un 2$-match, and 1 position otherwise. Once the top row is filled, we consider filling the rest of $\la$ by remove the top row and the column containing $\times$. The remaining cells form a shape $\bar \la$. Thus, the equation above is proved.

For the pattern $\un2 1$, a similar reasoning can also be applied. The filling of $\la$ has a $\un2 1$-match at the position of $\times$ in the topmost row if and only if $\times$ is not in the {\em leftmost} possible cell. Hence, we have

$$\sum_{\pi \in S_\la} x^{pmp_{\un2 1}(\pi)}  = (1+ (\la_k-1)x)\sum_{\pi \in S_{\bar \la}} x^{pmp_{\un2 1}(\pi)}.$$

Therefore, we proved the lemma.

\end{proof}

Here, we prove the main theorem:

\begin{proof}{(of Theorem \ref{lengthn})}
Note that $S_n$ is a disjoint union of $S_n^Q$ for all $Q$ such that $S_n^Q \neq \emptyset$. So, we have

\begin{align*}
P_{n,P_1}(x) &= \sum_{\sigma \in S_n} x^{pmp_{P_1}(\sigma)} & &\\
&= \sum_Q \sum_{\sigma \in S^Q_n} x^{pmp_{P_1}(\sigma)}  & & \\
&= \sum_Q \sum_{\sigma \in S_{\la(Q)}} x^{pmp_{1\un2}(\sigma)}& & \text{By Lemma \ref{bite}}\\
&= \sum_Q \sum_{\sigma \in S_{\la(Q)}} x^{pmp_{\un21}(\sigma)} & & \text{By Lemma \ref{equivla1221}} \\
&= \sum_Q \sum_{\sigma \in S^Q_n} x^{pmp_{P_2}(\sigma)}  & & \text{By Lemma \ref{bite}} \\
&= \sum_{\sigma \in S_n} x^{pmp_{P_2}(\sigma)} = P_{n,P_2}(x) & & 
\end{align*}

Thus, $P_1$ and $P_2$ are $PMP$-Wilf-equivalent.

To see that $P_1 = 1\un2 p_1\ldots p_{n-2}$ and $P_3 = 2 \un 1  p_1\ldots p_{n-2}$, Let $P_1^{inv}$ and $P_3^{inv}$ be positional marked patterns obtained from $P_1$ and $P_3$ by reflecting diagram of $P_1$ and $P_3$ along the diagonal $y=x$ respectively. By Lemma \ref{sym}, $P_1$ and $P_1^{inv}$ are $PMP$-Wilf-equivalent, and so are $P_3$ and $P_3^{inv}$. However, $P_1^{inv}$ has the form $1 \un 2 q_1 \ldots q_{n-2}$ and $P_3^{inv}$ has the form of $\un2 1 q_1 \ldots q_{n-2}$, thus $P_1^{inv}$ and $P_3^{inv}$ are $PMP$-Wilf-equivalent by the main argument above. Therefore, $P_1$ and $P_3$ are $PMP$-Wilf-equivalent.

\end{proof}

\noindent {\bf Remark.} Alternately, we can prove that $P_2$ and $P_3$ are $PMP$-Wilf-equivalent by considering the rightmost column of any Ferrers board $\la$.

To end this section, we state possible nontrivial $PMP$-Wilf-equivalence classes of size 4. According to numerical data, we conjecture that $\un 1234 $, $\un1243$ and $\un1432$ are equivalent. Also, $134\un2$ and $\un2413$ are equivalent. 
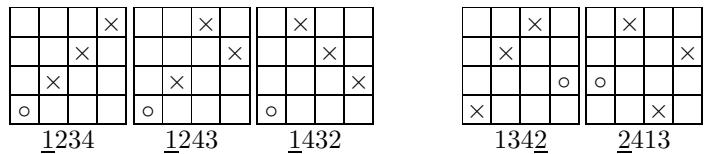
\begin{center}
\shortstack{$\young(\hfil \hfil \hfil \times,\hfil \hfil \times \hfil,\hfil \times \hfil \hfil,\circ \hfil \hfil \hfil )$ \\ $\un 1234 $} 
\shortstack{$\young(\hfil \hfil \times \hfil  ,\hfil \hfil \hfil \times,\hfil \times \hfil \hfil,\circ \hfil \hfil \hfil )$ \\ $\un1243$} 
\shortstack{$\young(\hfil \times \hfil \hfil  ,\hfil \hfil \times \hfil,\hfil \hfil \hfil \times,\circ \hfil \hfil \hfil )$ \\ $\un1432$}
\hspace{1cm}
\shortstack{$\young(\hfil \hfil \times \hfil  ,\hfil \times \hfil \hfil,\hfil \hfil \hfil \circ,\times \hfil \hfil \hfil )$ \\ $134\un2$} 
\shortstack{$\young(\hfil \times \hfil \hfil  ,\hfil \hfil \hfil \times,\circ \hfil \hfil \hfil,\hfil \hfil \times \hfil )$ \\ $\un2413$} 
\end{center}

\end{section}

\begin{section}{Future research}\label{future}
One way to generalize positional marked pattern is to consider multiple patterns. Given a collection of positional marked patterns $\Gamma$ and $\sigma \in S_n$, we say that $\sigma$ has a $\Gamma$-match at position $\ell$ if $\sigma$ has a $\tau$-match at position $\ell$ for some $\tau \in \Gamma$. Given $\sigma \in S_n$, let $pmp_\Gamma(\sigma)$ denote the number of positions $\ell$ such that $\sigma$ has a $\Gamma$-match at position $\ell$. Let $P_{n,\Gamma}(x) = \sum_{\sigma \in S_n} x^{pmp_\Gamma(\sigma)}$. Given two collections of positional marked patterns $\Gamma_1$ and $\Gamma_2$, they are Wilf-equivalent if $P_{n,\Gamma_1}(x) = P_{n,\Gamma_2}(x)$. This is a nice generalization to positional marked patterns since the constant term of $P_{n,\Gamma}(x)$ enumerates the number of $\sigma \in S_n$ avoiding all patterns in $\Gamma$, which replicates what $P_{n,\tau}(x)$ provides in single pattern cases. Here, we stated, without proof, a result on collection of positional marked patterns similar to Theorem \ref{lengthn}. The proof also follows the same logic as in the proof of Theorem $\ref{lengthn}$.

\begin{theorem}
Let $\tau_1,\tau_2,\ldots,\tau_\ell$ be rearrangements of $\set{3,4,\ldots,k}$, so that $1\un2\tau_i$ and $\un21\tau_i$ are elements of $S_k^*$. Let $\Gamma_1 = \set{ 1\un2 \tau_i | 1 \le i \le \ell}$ and $\Gamma_2 = \set{ \un21 \tau_i | 1 \le i \le \ell}$. Then $\Gamma_1$ and $\Gamma_2$ are Wilf-equivalent.
\end{theorem}

Moreover, $\Gamma$-matching will let us realize positional marked pattern as a refinement of marked mesh pattern defined by Kitaev and Remmel \cite{KR}. Given $a,b,c,d \in \Z_{\ge 0}$, and given $\sigma \in S_n$, we say $\sigma$ matches $MMP(a,b,c,d)$ at position $l$ if in the diagram of $\sigma$ relative to the coordinate system which has $\times$ in the $l$-th column as its origin there are at least $a$ $\times$'s in quadrant I, at least $b$ $\times$'s in quadrant II, at least $c$ $\times$'s in quadrant III, and  at least $d$ $\times$'s in quadrant IV. As an example, $\sigma = 25417683$ matches $MMP(3,0,1,1)$ at position 3:

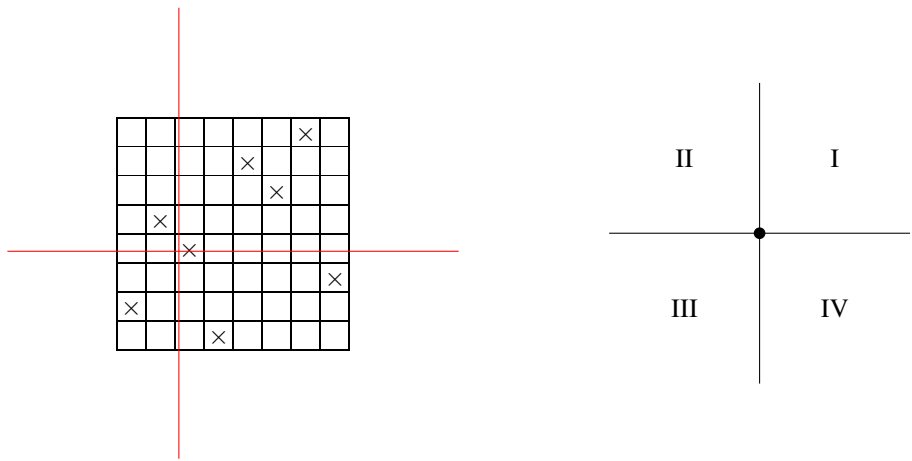
\begin{center}
\begin{tikzpicture}

\node (0,0) {\young(\hfil\hfil\hfil\hfil\hfil\hfil\times\hfil
,\hfil\hfil\hfil\hfil\times\hfil\hfil\hfil
,\hfil\hfil\hfil\hfil\hfil\times\hfil\hfil
,\hfil\times\hfil\hfil\hfil\hfil\hfil\hfil
,\hfil\hfil\times\hfil\hfil\hfil\hfil\hfil
,\hfil\hfil\hfil\hfil\hfil\hfil\hfil\times
,\times\hfil\hfil\hfil\hfil\hfil\hfil\hfil
,\hfil\hfil\hfil\times\hfil\hfil\hfil\hfil
)};

\draw[red] (-3,-0.24) -- (3,-0.24);
\draw[red] (-0.72,3) -- (-0.72,-3);

\draw (5,0) -- (9,0);
\draw (7,-2) -- (7,2);
\filldraw (7,0) circle (2pt);
\node at (8,1) {I};
\node at (6,1) {II};
\node at (6,-1) {III};
\node at (8,-1) {IV};

\end{tikzpicture}.
\end{center}

It is easy to see that, for any $a,b,c,d$, $MMP(a,b,c,d)$ is equivalent to $\Gamma$-matching for some collection of positional marked patterns $\Gamma$. For example, to match $MMP(2,0,0,0)$ at some position, it is the same as to match $\Gamma = \set{\un123, \un132}$ at the same position. Therefore, as we introduce multiple positional marked patterns, one can realize positional marked patterns as a refinement of marked mesh patterns.

\end{section}




\bibliographystyle{abbrvnat}
\bibliography{bibtex}

\begin{thebibliography}{10}
\providecommand{\natexlab}[1]{#1}
\providecommand{\url}[1]{\texttt{#1}}
\expandafter\ifx\csname urlstyle\endcsname\relax
  \providecommand{\doi}[1]{doi: #1}\else
  \providecommand{\doi}{doi: \begingroup \urlstyle{rm}\Url}\fi

\bibitem[Babson and West(2000)]{BW}
E.~Babson and J.~West.
\newblock The permutations $123p_4\ldots p_m$ and $321p_4\ldots p_m$ are
  {W}ilf-equivalent.
\newblock \emph{Graphs and Combinatorics}, 16:173:\penalty0 373--380, 2000.

\bibitem[Br\"and\'en and Claesson(2011)]{BC}
P.~Br\"and\'en and A.~Claesson.
\newblock Mesh patterns and the expansion of permutation statistics as sums of
  permutation patterns.
\newblock \emph{Electronic J. Combin.}, 18:\penalty0 2, 2011.

\bibitem[Davis(2015)]{MD}
M.~Davis.
\newblock Quadrant marked mesh patterns and the $r$-stirling numbers.
\newblock \emph{Journal of Integer Sequences}, 18, 2015.

\bibitem[Kitaev and Remmel(2012{\natexlab{a}})]{KR}
S.~Kitaev and J.~Remmel.
\newblock Quadrant marked mesh patterns.
\newblock \emph{Journal of Integer Sequence}, 15 issue 4\penalty0
  (12.4.7):\penalty0 29, 2012{\natexlab{a}}.

\bibitem[Kitaev and Remmel(2012{\natexlab{b}})]{KR2}
S.~Kitaev and J.~Remmel.
\newblock Quadrant marked mesh patterns in alternating permutations.
\newblock \emph{Sem. Lothar. Combin}, B68a:\penalty0 20, 2012{\natexlab{b}}.

\bibitem[Kitaev and Remmel(2013)]{KR3}
S.~Kitaev and J.~Remmel.
\newblock Quadrant marked mesh patterns in alternating permutations ii.
\newblock \emph{Journal of Combinatorics}, 4 no 1:\penalty0 31--65, 2013.

\bibitem[Kitaev et~al.(2012)Kitaev, Remmel, and Tiefenbruck]{KRT}
S.~Kitaev, J.~Remmel, and M.~Tiefenbruck.
\newblock Marked mesh patterns in 132-avoiding permutations.
\newblock \emph{Pure Mathematics and Applications}, 23:\penalty0 219--256,
  2012.

\bibitem[Kitaev et~al.(2015)Kitaev, Remmel, and Tiefenbruck]{KRT2}
S.~Kitaev, J.~Remmel, and M.~Tiefenbruck.
\newblock Marked mesh patterns in 132-avoiding permutations ii.
\newblock \emph{Integers: Electronic Journal of Combinatorial Number Theory},
  A16:\penalty0 33, 2015.

\bibitem[Qiu and Remmel(2018)]{QR}
D.~Qiu and J.~B. Remmel.
\newblock {Quadrant marked mesh patterns in 123-avoiding permutations}.
\newblock \emph{{DMTCS}}, {Vol. 19 no. 2, Permutation Patterns 2016}, 2018.
\newblock \doi{10.23638/DMTCS-19-2-12}.
\newblock URL \url{https://dmtcs.episciences.org/4735}.

\bibitem[\'Ulfarsson(2010)]{U}
H.~A. \'Ulfarsson.
\newblock A unification of permutation patterns related to schubert varieties.
\newblock \emph{DMTCS proc.}, AN:\penalty0 1057--1068, 2010.

\end{thebibliography}
\label{sec:biblio}






\end{document}